\documentclass{article}
\usepackage{amsmath}
\usepackage{amsfonts}
\usepackage{amssymb}
\usepackage{amscd}

\newtheorem{theorem}{Theorem}[section]

\newtheorem{condition}[theorem]{Condition}

\newtheorem{corollary}{Corollary}
\newtheorem{criterion}{Criterion}[section]
\newtheorem{definition}{Definition}[section]

\newtheorem{lemma}{Lemma}[section]

\newtheorem{problem}{Problem}[section]
\newtheorem{proposition}{Proposition}[section]

\newenvironment{proof}[1][Proof]{\noindent\textbf{#1.} }{\ \rule{0.5em}{0.5em}}

\input{tcilatex}
\begin{document}

\title{Automorphisms of the category of finitely generated free groups of
the some subvariety of the variety of all groups}
\author{R. Barbosa Fernandes \and A. Tsurkov \\
Mathematical Department, CCET,\\
Federal University of Rio Grande do Norte (UFRN),\\
Av. Senador Salgado Filho, 3000,\\
Campus Universit\'{a}rio, Lagoa Nova, \\
Natal - RN - Brazil - CEP 59078-970,\\
ruanbarbosafernandes@gmail.com,\\
arkady.tsurkov@gmail.com}
\maketitle

\begin{abstract}
In universal algebraic geometry the category $\Theta ^{0}$ of the finite
generated free algebras of some fixed variety $\Theta $ of algebras and the
quotient group $\mathfrak{A/Y}$ are very important. Here $\mathfrak{A}$ is a
group of all automorphisms of the category $\Theta ^{0}$ and $\mathfrak{Y}$
is a group of all inner automorphisms of this category.

In the varieties of all the groups, all the abelian groups \cite%
{PlotkinZhitom}, all the nilpotent groups of the class no more then $n$ ($%
n\geq 2$) \cite{TsurkovNilpotent}\ the group $\mathfrak{A/Y}$ is trivial. B.
Plotkin posed a question: "Is there a subvariety of the variety of all the
groups, such that the group $\mathfrak{A/Y}$ in this subvariety is not
trivial?" A. Tsurkov hypothesized that exist some varieties of periodic
groups, such that the groups $\mathfrak{A/Y}$ in these varieties is not
trivial. In this paper we give an example of one subvariety of this
kind.\medskip

\medskip

\end{abstract}

\pagebreak \hangindent=5.5cm \hangafter=0 \noindent \textit{Poetry\newline
is like mining radium.\newline
For every gram\newline
you work a year.\newline
For the sake of a single word\newline
you waste\newline
a thousand tons\newline
of verbal ore.}\newline
V. V. Mayakovsky

\section{Introduction\label{Intr}}

\setcounter{equation}{0}

The paper is devoted to some aspects of universal algebraic geometry, i.e.,
geometry over \textit{universal algebras} (for definition of universal
algebra\textit{\ }see, for example \cite[Chapter 3, 1. 3]{KUROSH}). In fact,
universal algebra is the set with the some list (signature) of operations.
We will say shortly "algebra" instead "universal algebra".

All definitions of the basic notions of the universal algebraic geometry can
be found, for example, in \cite{PlotkinVarCat}, \cite{PlotkinNotions}, \cite%
{PlotkinSame} and \cite{PP}. Also, there are fundamental papers \cite{BMR}, 
\cite{MR} and \cite{DMR2}, \cite{DMR5}.

One of the natural question of universal algebraic geometry is as follows:

\begin{problem}
\label{pr:1} When do two algebras $H_{1}$ and $H_{2}$ from the some variety
of algebras $\Theta $ have the same algebraic geometry?
\end{problem}

Under the sameness of geometries over $H_{1}$ and $H_{2}$ we mean an
isomorphism of the categories of algebraic sets over $H_{1}$ and $H_{2}$,
respectively. So, Problem \ref{pr:1} is ultimately related to the following
one:

\begin{problem}
\label{pr:2} What are the conditions which provide an isomorphism of the
categories of algebraic sets over the algebras $H_{1}$ and $H_{2}$?
\end{problem}

Notions of geometric and automorphic equivalences of algebras play here a
crucial role.

In universal algebraic geometry we consider some variety $\Theta $ of
universal algebras of the signature $\Omega $. We denote by $X_{0}$ an
infinite countable set of symbols. By $\mathfrak{F}\left( X_{0}\right) $ we
denote the set of all finite subsets of $X_{0}$. We will consider the
category $\Theta ^{0}$, whose objects are all free algebras $F\left(
X\right) $ of the variety $\Theta $ generated by finite subsets $X\in 
\mathfrak{F}\left( X_{0}\right) $. Morphisms of the category $\Theta ^{0}$
are homomorphisms of such algebras. We will occasionally denote $F\left(
X\right) =F\left( x_{1},x_{2},\ldots ,x_{n}\right) $ if $X=\left\{
x_{1},x_{2},\ldots ,x_{n}\right\} $.

We consider a system of equations $T\subseteq F\times F$, where $F\in 
\mathrm{Ob}\Theta ^{0}$, and we solve these equations in arbitrary algebra $%
H\in \Theta $.

The set $\mathrm{Hom}\left( F,H\right) $ serves as an affine space over the
algebra $H$: the solution of the system $T$ is a homomorphism $\mu \in 
\mathrm{Hom}\left( F,H\right) $ such that $\mu \left( t_{1}\right) =\mu
\left( t_{2}\right) $ holds for every $\left( t_{1},t_{2}\right) \in T$ or $%
T\subseteq \ker \mu $. $T_{H}^{\prime }=\left\{ \mu \in \mathrm{Hom}\left(
F,H\right) \mid T\subseteq \ker \mu \right\} $ will be the set of all the
solutions of the system $T$. We call these sets \textit{algebraic}, as in
the classical algebraic geometry.

For every set of points $R\subseteq \mathrm{Hom}\left( F,H\right) $ we
consider a congruence of equations defined in this way: $R_{H}^{\prime
}=\bigcap\limits_{\mu \in R}\ker \mu $. This is a maximal system of
equations which has the set of solutions $R$. For every set of equations $T$
we consider its algebraic closure $T_{H}^{\prime \prime
}=\bigcap\limits_{\mu \in T_{H}^{\prime }}\ker \mu $ with respect to the
algebra $H$. A set $T\subseteq F\times F$ is called $H$-closed if $%
T=T_{H}^{\prime \prime }$. An $H$-closed set is always a congruence. We
denote the family of all $H$-closed congruences in $F$ by $Cl_{H}(F)$.

\begin{definition}
Algebras $H_{1},H_{2}\in \Theta $ are \textbf{geometrically equivalent} if
and only if for every $F\in \mathrm{Ob}\Theta ^{0}$ and every $T\subseteq
F\times F$ the equality $T_{H_{1}}^{\prime \prime }=T_{H_{2}}^{\prime \prime
}$ is fulfilled.
\end{definition}

By this definition, algebras $H_{1},H_{2}\in \Theta $ are geometrically
equivalent if and only if the families $Cl_{H_{1}}(F)$ and $Cl_{H_{2}}(F)$
coincide for every $F\in \mathrm{Ob}\Theta ^{0}$.

\begin{definition}
\label{Autom_equiv}\cite{PlotkinSame}We say that \textit{algebras }$%
H_{1},H_{2}\in \Theta $\textit{\ are \textbf{automorphically equivalent} if
there exist an automorphism }$\Phi :\Theta ^{0}\rightarrow \Theta ^{0}$%
\textit{\ and the bijections}%
\begin{equation*}
\alpha (\Phi )_{F}:Cl_{H_{1}}(F)\rightarrow Cl_{H_{2}}(\Phi (F))
\end{equation*}%
for every $F\in \mathrm{Ob}\Theta ^{0}$, \textit{coordinated in the
following sense: if }$F_{1},F_{2}\in \mathrm{Ob}\Theta ^{0}$\textit{, }$\mu
_{1},\mu _{2}\in \mathrm{Hom}\left( F_{1},F_{2}\right) $\textit{, }$T\in
Cl_{H_{1}}(F_{2})$\textit{\ then}%
\begin{equation*}
\tau \mu _{1}=\tau \mu _{2},
\end{equation*}%
\textit{if and only if }%
\begin{equation*}
\widetilde{\tau }\Phi \left( \mu _{1}\right) =\widetilde{\tau }\Phi \left(
\mu _{2}\right) ,
\end{equation*}%
\textit{where }$\tau :F_{2}\rightarrow F_{2}/T$\textit{, }$\widetilde{\tau }%
:\Phi \left( F_{2}\right) \rightarrow \Phi \left( F_{2}\right) /\alpha (\Phi
)_{F_{2}}\left( T\right) $\textit{\ are the natural epimorphisms.}
\end{definition}

The definition of the automorphic equivalence in the language of the
category of coordinate algebras was considered in \cite{PlotkinSame} and 
\cite{TsurkovManySorted}. Intuitively we can say that algebras $%
H_{1},H_{2}\in \Theta $ are automorphically equivalent if and only if the
families $Cl_{H_{1}}(F)$ and $Cl_{H_{2}}(\Phi \left( F\right) )$ coincide up
to a changing of coordinates. This changing is defined by the automorphism $%
\Phi $.

\begin{definition}
\label{inner}An automorphism $\Upsilon $ of an arbitrary category $\mathfrak{%
K}$ is \textbf{inner}, if it is isomorphic as a functor to the identity
automorphism of the category $\mathfrak{K}$.
\end{definition}

It means that for every $F\in \mathrm{Ob}\mathfrak{K}$ there exists an
isomorphism $\sigma _{F}^{\Upsilon }:F\rightarrow \Upsilon \left( F\right) $
such that for every $\mu \in \mathrm{Mor}_{\mathfrak{K}}\left(
F_{1},F_{2}\right) $%
\begin{equation*}
\Upsilon \left( \mu \right) =\sigma _{F_{2}}^{\Upsilon }\mu \left( \sigma
_{F_{1}}^{\Upsilon }\right) ^{-1}
\end{equation*}%
\noindent holds. It is clear that the set $\mathfrak{Y}$ of all inner
automorphisms of an arbitrary category $\mathfrak{K}$ is a normal subgroup
of the group $\mathfrak{A}$ of all automorphisms of this category.

If an inner automorphism $\Upsilon $ provides the automorphic equivalence of
the algebras $H_{1}$ and $H_{2}$, where $H_{1},H_{2}\in \Theta $, then $%
H_{1} $ and $H_{2}$ are geometrically equivalent (see \cite[Proposition 9]%
{PlotkinSame}). Therefore the quotient group $\mathfrak{A/Y}$ measures the
possible difference between the geometric equivalence and automorphic
equivalence of algebras from the variety $\Theta $: if the group $\mathfrak{%
A/Y}$ is trivial, then the geometric equivalence and automorphic equivalence
coincide in the variety $\Theta $. The converse is not true. For example, in
the variety of the all linear spaces over some fixed field $k$ of
characteristic $0$ we have that $\mathfrak{A/Y}\cong \mathrm{Aut}k$, where $%
\mathrm{Aut}k$ is the group of all the automorphisms of the field $k$. The
proof of this fact can be achieved by the method of \cite{ATsurkovLinAlg}.
But all linear spaces over every fixed field $k$ are geometrically
equivalent. This fact is a simple conclusion from \cite[Theorem 3]{PPT}.

In the varieties of all the groups, all the abelian groups \cite%
{PlotkinZhitom}, all the nilpotent groups of the class no more then $n$ ($%
n\geq 2$) \cite{TsurkovNilpotent}\ the group $\mathfrak{A/Y}$ is trivial, so
the geometric equivalence and the automorphic equivalence coincide in these
varieties. B. Plotkin posed a question: "Is there a subvariety of the
variety of all the groups, such that the group $\mathfrak{A/Y}$ in this
subvariety is not trivial?" A. Tsurkov hypothesized that exist some
varieties of periodic groups, such that the groups $\mathfrak{A/Y}$ in these
varieties is not trivial. In this article, we confirm this hypothesis.

We consider a subvariety $\Theta $ of the variety of all groups. Our
subvariety is defined by identities%
\begin{equation}
x^{4}=1,  \label{exponent}
\end{equation}%
\begin{equation}
\left( \left( x_{1},x_{2}\right) ,\left( x_{3},x_{4}\right) \right) =1,
\label{metab}
\end{equation}%
and%
\begin{equation}
\left( \left( \left( \left( x_{1},x_{2}\right) ,x_{3}\right) ,x_{4}\right)
,x_{5}\right) =1,  \label{4nilp}
\end{equation}%
in other words, this is a variety of all nilpotent class no more then $4$,
metabelian and Sanov \cite{Sanov} groups. We will use the method of the
verbal operations elaborated in \cite{PlotkinZhitom} for the calculation of
the quotient group $\mathfrak{A/Y}$ for the variety $\Theta $. In the next
Section we will explain this method.

\section{Method of verbal operations}

\setcounter{equation}{0}

In this section we will explain the method of the verbal operations for the
computing of the quotient group $\mathfrak{A/Y}$ in the case of arbitrary
variety $\Theta $ of universal algebras of the signature $\Omega $. The
reader also can see the explanation and application of this method in \cite%
{PlotkinZhitom}, \cite{TsurAutomEqAlg}, \cite{TsurkovNilpotent}, \cite%
{TsurkovManySorted} and \cite{TsurkovClassicalVar}.

\subsection{First definitions and basic facts}

This method we can apply only if the following condition holds in the
variety $\Theta $:

\begin{condition}
\label{monoiso}\cite{PlotkinZhitom}$\Phi \left( F\left( x\right) \right)
\cong F\left( x\right) $ for every automorphism $\Phi $ of the category $%
\Theta ^{0}$ for every $x\in X_{0}$.
\end{condition}

In this case, by \cite[Theorem 2.1]{TsurkovManySorted}, for every $\Phi \in 
\mathfrak{A}$ there exists a system of bijections%
\begin{equation}
S=\left\{ s_{F}:F\rightarrow \Phi \left( F\right) \mid F\in \mathrm{Ob}%
\Theta ^{0}\right\} ,  \label{bij_system}
\end{equation}%
such that for every $\psi \in \mathrm{Mor}_{\Theta ^{0}}\left( A,B\right) $
the diagram%
\begin{equation*}
\begin{CD} A@>>{s _{A}}>{\Phi \left( A\right)} \\ @VV{\psi}V @V{\Phi \left(
\psi \right)}VV \\ B@>{s _{B}}>>{\Phi \left( B\right)} \\ \end{CD}
\end{equation*}%
\noindent is commutative. It means that $\Phi $\ acts on the morphisms $\psi
:A\rightarrow B$ of $\Theta ^{0}$\ as follows:\textit{\ }%
\begin{equation}
\Phi \left( \psi \right) =s_{B}\psi s_{A}^{-1}.  \label{acting}
\end{equation}

\begin{definition}
\label{conected_toaut}We say that the system of bijections (\ref{bij_system}%
) is a system of bijections \textbf{associated with the automorphism} $\Phi
\in \mathfrak{A}$ if this system fulfills the condition (\ref{acting}).
\end{definition}

One automorphism of the category $\Theta ^{0}$ in general can be associated
with various systems of bijections and some system of bijections can be
associated with various automorphisms.

In \cite{PlotkinZhitom} the notion of the strongly stable automorphism of
the category $\Theta ^{0}$ was defined:

\begin{definition}
\label{str_stab_aut}\textit{An automorphism $\Phi $ of the category }$\Theta
^{0}$\textit{\ is called \textbf{strongly stable} if it satisfies the
conditions:}

\begin{enumerate}
\item $\Phi $\textit{\ preserves all objects of }$\Theta ^{0}$\textit{,}

\item there exists one system of bijections associated with the automorphism 
$\Phi $\textit{\ such that}%
\begin{equation}
s_{F}\mid _{X}=id_{X}  \label{stab_bij}
\end{equation}%
\textit{\ }holds for every $F\left( X\right) \in \mathrm{Ob}\Theta ^{0}$.
\end{enumerate}
\end{definition}

In other words, we can say that an automorphism of the category $\Theta ^{0}$%
\ is called strongly stable if it preserves all objects of $\Theta ^{0}$ and
there is some system of bijections associated with this automorphism such
that all the bijections of this system preserve all generators of domains.

It is clear that the set $\mathfrak{S}$ of all strongly stable automorphisms
of the category $\Theta ^{0}$ is a subgroup of the group $\mathfrak{A}$ of
all automorphisms of this category. By \cite[Theorem 2.3]{TsurkovManySorted}%
, $\mathfrak{A=YS}$ holds if in the category $\Theta ^{0}$ fulfills the
Condition \ref{monoiso}. In this case we have that $\mathfrak{A/Y\cong
S/S\cap Y}$. So to study $\mathfrak{A/Y}$ we must compute the groups $%
\mathfrak{S}$ and $\mathfrak{S\cap Y}$.

\subsection{Strongly stable automorphism and strongly stable system of
bijections\label{automorphism_bijections}}

We consider the strongly stable automorphism $\Phi \in \mathfrak{S}$. There
exists a system of bijections associated with this automorphism which is a
subject of Definition \ref{str_stab_aut}. This system of bijections is
uniquely defined by the automorphism $\Phi $, because the equality $%
s_{A}\left( a\right) =\Phi \left( \alpha \right) \left( x\right) $ holds for
every $A\in \mathrm{Ob}\Theta ^{0}$ and every $a\in A$, where $\alpha
:F\left( x\right) \rightarrow A$ is a homomorphism defined by $\alpha \left(
x\right) =a$ (see \cite[Proposition 3.1]{TsurkovManySorted}). We denote this
system of bijections by $S_{\Phi }$, and its bijections we denote by $%
s_{F}^{\Phi }$ for every $F\in \mathrm{Ob}\Theta ^{0}$.

\begin{definition}
\label{sss}The system of bijections $S=\left\{ s_{F}:F\rightarrow F\mid F\in 
\mathrm{Ob}\Theta ^{0}\right\} $ is called \textbf{strongly stable} if for
every $A,B\in \mathrm{Ob}\Theta ^{0}$ and every $\mu \in \mathrm{Mor}%
_{\Theta ^{0}}\left( A,B\right) $ the mappings $s_{B}\mu s_{A}^{-1}$, $%
s_{B}^{-1}\mu s_{A}:A\rightarrow B$ are homomorphisms and the condition (\ref%
{stab_bij}) are fulfilled.
\end{definition}

The set of all the strongly stable system of bijections we denote by $%
\mathcal{SSSB}$.

It is clear that system of bijections $S_{\Phi }$ is strongly stable. Hence
the mapping $\mathcal{A}:\mathfrak{S}\rightarrow \mathcal{SSSB}$ such that $%
\mathcal{A}\left( \Phi \right) =S_{\Phi }$ is well defined by \cite[%
Proposition 3.1]{TsurkovManySorted}. This mapping is one to one and onto by 
\cite[Proposition 3.2]{TsurkovManySorted}.

If $\Phi _{1},\Phi _{2}\in \mathfrak{S}$ then there are strongly stable
systems of bijections 
\begin{equation*}
\mathcal{A}\left( \Phi _{1}\right) =S_{\Phi _{1}}=\left\{ s_{F}^{\Phi
_{1}}:F\rightarrow F\mid F\in \mathrm{Ob}\Theta ^{0}\right\}
\end{equation*}%
and%
\begin{equation*}
\mathcal{A}\left( \Phi _{2}\right) =S_{\Phi _{2}}=\left\{ s_{F}^{\Phi
_{2}}:F\rightarrow F\mid F\in \mathrm{Ob}\Theta ^{0}\right\} .
\end{equation*}%
For every $\psi \in \mathrm{Mor}_{\Theta ^{0}}\left( F_{1},F_{2}\right) $
the equality $\Phi _{2}\Phi _{1}\left( \psi \right) =s_{F_{2}}^{\Phi
_{2}}s_{F_{2}}^{\Phi _{1}}\psi \left( s_{F_{1}}^{\Phi _{1}}\right)
^{-1}\left( s_{F_{1}}^{\Phi _{2}}\right) ^{-1}$. It means that the system of
bijections%
\begin{equation*}
\left\{ s_{F}^{\Phi _{2}}s_{F}^{\Phi _{1}}:F\rightarrow F\mid F\in \mathrm{Ob%
}\Theta ^{0}\right\}
\end{equation*}%
is associated with the automorphism $\Phi _{2}\Phi _{1}$. But it is clear
that this system is strongly stable, so this system of bijections is
uniquely defined strongly stable system of bijections corresponds to the
strongly stable automorphism $\Phi _{2}\Phi _{1}$, in other words,%
\begin{equation*}
\mathcal{A}\left( \Phi _{2}\Phi _{1}\right) =\left\{ s_{F}^{\Phi
_{2}}s_{F}^{\Phi _{1}}:F\rightarrow F\mid F\in \mathrm{Ob}\Theta
^{0}\right\} .
\end{equation*}

\subsection{Strongly stable system of bijections and applicable systems of
words\label{bijections_words}}

We consider the algebra $F=F\left( x_{1},\ldots ,x_{n}\right) \in \mathrm{Ob}%
\Theta ^{0}$ and take a word (element) $w=w\left( x_{1},\ldots ,x_{n}\right)
\in F\left( x_{1},\ldots ,x_{n}\right) $.

\begin{definition}
The operation $\omega ^{\ast }$: $\omega ^{\ast }\left( h_{1},\ldots
,h_{n}\right) =w\left( h_{1},\ldots ,h_{n}\right) $ is called \textbf{verbal
operation} defined on the algebra $H$ by the word $w$, where\linebreak $%
h_{i}\in H$, $1\leq i\leq n$, and $H\in \Theta $ is an arbitrary algebra of
the variety $\Theta $.
\end{definition}

The reader can compare this definition with the definition of word maps, 
\cite{Se}, \cite{KKP} and references therein.

Denote the signature of our variety $\Theta $ by $\Omega $. For every $%
\omega \in \Omega $ which has an arity $\rho _{\omega }$ we consider the
algebra $F_{\omega }=F\left( x_{1},\ldots ,x_{\rho _{\omega }}\right) \in 
\mathrm{Ob}\Theta ^{0}$. Having a system of words $W=\left\{ w_{\omega }\mid
\omega \in \Omega \right\} $ where $w_{\omega }\in F_{\omega }$, denote by $%
H_{W}^{\ast }$ the algebra which coincides with $H$ as a set, but instead of
the original operations $\left\{ \omega \mid \omega \in \Omega \right\} $ it
possesses the system of the operations $\left\{ \omega ^{\ast }\mid \omega
\in \Omega \right\} $ where $\omega ^{\ast }$ is a verbal operation defined
by word $w_{\omega }$.

{We can consider the algebras $H$ and $H_{W}^{\ast }$ as algebras with the
same signature $\Omega $: the realization of the operation $\omega \in
\Omega $ in the algebra $H$ is the operation $\omega $ and the realization
of the operation $\omega \in \Omega $ in the algebra $H_{W}^{\ast }$ is the
operation $\omega ^{\ast }$. So, if $A$ and $B$ are algebras with the
original operations $\left\{ \omega \mid \omega \in \Omega \right\} $, $%
A_{W}^{\ast }$ and $B_{W}^{\ast }$ are algebras with the operations $\left\{
\omega ^{\ast }\mid \omega \in \Omega \right\} $, we can consider the
homomorphisms from $A$ to $B_{W}^{\ast }$, from $A_{W}^{\ast }$ to $B$ and
so on.}

\begin{definition}
\label{asw}The system of words $W=\left\{ w_{\omega }\mid \omega \in \Omega
\right\} $ is called \textbf{applicable} if $w_{\omega }\left( x_{1},\ldots
,x_{\rho _{\omega }}\right) \in F_{\omega }$ and for every $F=F\left(
X\right) \in \mathrm{Ob}\Theta ^{0}$ there exists an isomorphism $%
s_{F}:F\rightarrow F_{W}^{\ast }$ such that $s_{F}\mid _{X}=id_{X}$.
\end{definition}

The set of all the applicable systems of words we denote by $\mathcal{ASW}$.
This set is never empty. The trivial example of the applicable system of
words, which always exists, give as the system $W=\left\{ w_{\omega }\mid
\omega \in \Omega \right\} $, such that $w_{\omega }=\omega $ for every $%
\omega \in \Omega $.

We suppose that $W=\left\{ w_{\omega }\mid \omega \in \Omega \right\} $ is
an applicable system of words and consider the system of isomorphisms $%
S=\left\{ s_{F}:F\rightarrow F_{W}^{\ast }\mid F\in \mathrm{Ob}\Theta
^{0}\right\} $ mentioned in Definition \ref{asw}. The isomorphism $s_{F}$ as
mapping from algebra $F\in \mathrm{Ob}\Theta ^{0}$ to itself is only a
bijection, which fulfill conditions (\ref{stab_bij}). The mappings $s_{B}\mu
s_{A}^{-1}$, $s_{B}^{-1}\mu s_{A}:A\rightarrow B$ are homomorphisms by \cite[%
Corollary 2 from Proposition 3.4]{TsurkovManySorted} for every $A,B\in 
\mathrm{Ob}\Theta ^{0}$ and every $\mu \in \mathrm{Mor}_{\Theta ^{0}}\left(
A,B\right) $. So $S=\left\{ s_{F}:F\rightarrow F\mid F\in \mathrm{Ob}\Theta
^{0}\right\} $ is a strongly stable\ system of bijections. From \cite[%
Proposition 3.5]{TsurkovManySorted} we conclude that the isomorphisms $%
s_{F}:F\rightarrow F_{W}^{\ast }$ such that (\ref{stab_bij}) holds are
uniquely defined by the system of words $W$. So the system of bijections $S$
is uniquely defined by $W$. We denote this system by $S_{W}$. Therefore the
mapping $\mathcal{B}:\mathcal{ASW\rightarrow SSSB}$ such that $\mathcal{B}%
\left( W\right) =S_{W}$ is well defined. This mapping is one to one and onto
by \cite[Proposition 3.6]{TsurkovManySorted}. In particular, if system of
bijections $S=\left\{ s_{F}:F\rightarrow F\mid F\in \mathrm{Ob}\Theta
^{0}\right\} $ is a strongly stable system of bijections, then a word $%
w_{\omega }$ from the applicable system of words $W=\mathcal{B}^{-1}\left(
S\right) $ we can obtain by the formula%
\begin{equation}
w_{\omega }\left( x_{1},\ldots ,x_{\rho _{\omega }}\right) =s_{F_{\omega
}}\left( \omega \left( x_{1},\ldots ,x_{\rho _{\omega }}\right) \right) \in
F_{\omega },  \label{der_veb_opr}
\end{equation}%
where $\omega \in \Omega $ (see \cite[Susection 2.4]{PlotkinZhitom}, \cite[%
Equation (3.1)]{TsurkovManySorted}).

Now we can conclude \cite[Theorem 3.1]{TsurkovManySorted} that there is one
to one and onto correspondence $\mathcal{C}=\mathcal{B}^{-1}\mathcal{A}:%
\mathfrak{S}\rightarrow \mathcal{ASW}$. We denote $\mathcal{C}\left( \Phi
\right) $ by $W_{\Phi }$. The systems of words $W_{\Phi }$ is defined by
formula (\ref{der_veb_opr}) where bijections $s_{F_{\omega }}=s_{F_{\omega
}}^{\Phi }$ are the corresponding bijections of the system $\mathcal{A}%
\left( \Phi \right) =S_{\Phi }$.

Therefore we can calculate the group $\mathfrak{S}$ if we are able to find
all applicable system of words.

If $\Phi _{1},\Phi _{2}\in \mathfrak{S}$ \ and%
\begin{equation*}
\mathcal{A}\left( \Phi _{1}\right) =S_{\Phi _{1}}=\left\{ s_{F}^{\Phi
_{1}}:F\rightarrow F\mid F\in \mathrm{Ob}\Theta ^{0}\right\} ,
\end{equation*}%
\begin{equation*}
\mathcal{A}\left( \Phi _{2}\right) =S_{\Phi _{2}}=\left\{ s_{F}^{\Phi
_{2}}:F\rightarrow F\mid F\in \mathrm{Ob}\Theta ^{0}\right\}
\end{equation*}%
are strongly stable systems of bijections correspond to automorphisms $\Phi
_{1}$ and $\Phi _{2}$, then as we saw in the previous section, the strongly
stable system of bijections 
\begin{equation*}
\mathcal{A}\left( \Phi _{2}\Phi _{1}\right) =S=\left\{ s_{F}^{\Phi
_{2}}s_{F}^{\Phi _{1}}:F\rightarrow F\mid F\in \mathrm{Ob}\Theta ^{0}\right\}
\end{equation*}%
corresponds to the strongly stable automorphism $\Phi _{2}\Phi _{1}$. Hence,
by (\ref{der_veb_opr}), the applicable systems of words $\mathcal{B}%
^{-1}\left( S\right) =\mathcal{C}\left( \Phi _{2}\Phi _{1}\right) $ we can
obtain by formula%
\begin{equation}
w_{\omega }\left( x_{1},\ldots ,x_{\rho _{\omega }}\right) =s_{F_{\omega
}}^{\Phi _{2}}s_{F_{\omega }}^{\Phi _{1}}\left( \omega \left( x_{1},\ldots
,x_{\rho _{\omega }}\right) \right) ,  \label{der_veb_opr_prod}
\end{equation}%
where $\omega \in \Omega $.

\subsection{Automorphisms, which are strongly stable and inner}

For calculation of the group $\mathfrak{S\cap Y}$ we also have the following

\begin{criterion}
\label{inner_stable}\cite[Lemma 3]{PlotkinZhitom}The strongly stable
automorphism $\Phi $ of the category $\Theta ^{0}$, such that $\mathcal{C}%
\left( \Phi \right) =W_{\Phi }=W$, is inner if and only if for every $F\in 
\mathrm{Ob}\Theta ^{0}$ there exists an isomorphism $c_{F}:F\rightarrow
F_{W}^{\ast }$ such that%
\begin{equation}
c_{B}\psi =\psi c_{A}  \label{commutmor}
\end{equation}%
is fulfilled for every $A,B\in \mathrm{Ob}\Theta ^{0}$ and every $\psi \in 
\mathrm{Mor}_{\Theta ^{0}}\left( A,B\right) $.
\end{criterion}

Also we have

\begin{proposition}
\label{centr_func}\cite[Proposition 23]{GomesMessias}The system of
functions\linebreak $\left\{ c_{A}:A\rightarrow A\mid A\in \mathrm{Ob}\Theta
^{0}\right\} $ fulfills the equality (\ref{commutmor}) for every $A,B\in 
\mathrm{Ob}\Theta ^{0}$ and every $\psi \in \mathrm{Mor}_{\Theta ^{0}}\left(
A,B\right) $ if and only if there exists $c(x)\in F(x)$ such that%
\begin{equation}
c_{A}(a)=c(a),  \label{commutfunc}
\end{equation}%
for every $A\in \mathrm{Ob}\Theta ^{0}$ and every $a\in A$.
\end{proposition}

\begin{proof}
We consider $c(x)\in F(x)$ and define the system of functions\linebreak $%
\left\{ c_{A}:A\rightarrow A\mid A\in \mathrm{Ob}\Theta ^{0}\right\} $ by (%
\ref{commutfunc}). We have for every $\psi \in \mathrm{Mor}_{\Theta
^{0}}\left( A,B\right) $ and every $a\in A$ that the equality $\psi
c_{A}\left( a\right) =\psi \left( c\left( a\right) \right) =c\left( \psi
\left( a\right) \right) =c_{B}\psi \left( a\right) $ holds, because $\psi
\in \mathrm{Mor}_{\Theta ^{0}}\left( A,B\right) $.

We suppose that exists a system of functions $\left\{ c_{A}:A\rightarrow
A\mid A\in \mathrm{Ob}\Theta ^{0}\right\} $ which fulfills equality (\ref%
{commutmor}) for every $A,B\in \mathrm{Ob}\Theta ^{0}$ and every $\psi \in 
\mathrm{Mor}_{\Theta ^{0}}\left( A,B\right) $. We consider the algebra $%
F=F(x)\in \mathrm{Ob}\Theta ^{0}$. There exists $c(x)=c_{F}(x)\in F(x)$. For
every $A\in \mathrm{Ob}\Theta ^{0}$ and every $a\in A$ we can consider the
homomorphism $\alpha _{a}:F(x)\rightarrow A$, such that $\alpha _{a}\left(
x\right) =a$. Therefore $c_{A}(a)=c_{A}(\alpha _{a}\left( x\right) )=\alpha
_{a}\left( c_{F}(x)\right) =\alpha _{a}\left( c(x)\right) =c(a)$.
\end{proof}

\section{Application of the method of verbal operations}

\setcounter{equation}{0}

We consider every group as universal algebra with signature which has $3$
operations:%
\begin{equation*}
\Omega =\left\{ 1,-1,\cdot \right\} ,
\end{equation*}%
where the $0$-ary operation $1$ give us an unit of a group, the $1$-ary
operation $-1$ give us for an arbitrary element $g$ of a group $G$ the
inverse element $g^{-1}$ and the $2$-ary operation $\cdot $ give us for two
elements of a group $G$ its product.

The IBN (invariant basis number) property or invariant dimension property
was defined initially in the theory of rings and modules, see, for example, 
\cite[Definition 2.8]{Hungerford}. But then this concept was generalized to
arbitrary varieties of algebras:

\begin{definition}
\label{IBN}We say that the variety $\Theta $ has an \textbf{IBN property} if
for every $F_{\Theta }\left( X\right) ,F_{\Theta }\left( Y\right) \in 
\mathrm{Ob}\Theta ^{0}$ the $F_{\Theta }\left( X\right) \cong F_{\Theta
}\left( Y\right) $ holds if and only if $\left\vert X\right\vert =\left\vert
Y\right\vert $.
\end{definition}

By \cite{Fujiwara} our variety $\Theta $ has an IBN property. It is easy to
conclude from this fact that in the variety $\Theta $ the Condition \ref%
{monoiso} fulfills. So, the method of verbal operations is valid in our
variety.

Thus the strategy of our research is clear. First of all we will compute the 
$2$-generated free group of our variety $F_{\Theta }\left( x,y\right) $.

After that we will find all applicable system of words%
\begin{equation}
W=\left\{ w_{1},w_{-1}\left( x\right) ,w_{\cdot }\left( x,y\right) \right\} ,
\label{syst_words}
\end{equation}%
where $w_{1}$ is a constant which correspond to the $0$-ary operation $1$, $%
w_{-1}\left( x\right) \in F_{\Theta }\left( x\right) $ is a word which
correspond to the $1$-ary operation $-1$ and $w_{\cdot }\left( x,y\right)
\in F_{\Theta }\left( x,y\right) $ is a word which correspond to the $2$-ary
operation "$\cdot $". We will use the Definition \ref{asw} for the finding
of the applicable system of words. The necessary conditions for the system
of words to be applicable we will conclude from the fact that the
isomorphism $s_{F}:F\rightarrow F_{W}^{\ast }$, which exists for every $F\in 
\mathrm{Ob}\Theta ^{0}$, provide the fulfilling of all identities of the
variety $\Theta $ in the groups $F_{W}^{\ast }$. It will give us $4$ systems
of words of the form (\ref{syst_words}), which can be applicable.

In the next step of our research we will prove that all these systems of
words are applicable. We will prove that for all these system $W$ all
identities of the variety $\Theta $ really fulfill in the groups $%
F_{W}^{\ast }$ for every $F\in \mathrm{Ob}\Theta ^{0}$. This will allow us
to construct the homomorphism $s=s_{F\left( X\right) }:F\left( X\right)
\rightarrow \left( F\left( X\right) \right) _{W}^{\ast }$, such that $%
s_{\mid X}=id_{X}$ for every $F\left( X\right) \in \mathrm{Ob}\Theta ^{0}$.
After that we will find the inverse maps for every $s_{F\left( X\right) }$.
It allow to conclude that all homomorphisms $s_{F\left( X\right) }$ are
isomorphisms and all $4$ considered systems of words are applicable and
provide the strongly stable automorphisms of the category $\Theta ^{0}$.

And we will finish our research when we will compute for the category $%
\Theta ^{0}$ the group $\mathfrak{Y}\cap \mathfrak{S}$ by Criterion\textbf{\ 
}\ref{inner_stable} and Proposition \ref{centr_func}. We will see in the end
of our research that the group $\mathfrak{A}/\mathfrak{Y}$ of the category $%
\Theta ^{0}$ contains $2$ elements.

\section{Some properties of the varieties $\mathfrak{N}_{4}$ and $\Theta $}

\setcounter{equation}{0}

In this paper $\mathfrak{N}_{4}$ is the variety of nilpotent groups of class
no more then $4$. The free groups of this variety, generated by generators $%
x_{1},\ldots ,x_{n}$ we will denote by $N_{4}\left( x_{1},\ldots
,x_{n}\right) $.

We will denote $\left( \left( \left( \left( x,y\right) ,z\right) ,\ldots
\right) ,t\right) $ as $\left( x,y,z,\ldots ,t\right) $. Also we will denote
for every group $G$ the $\gamma _{1}\left( G\right) =G$ and $\gamma
_{i+1}\left( G\right) =\left( \gamma _{i}\left( G\right) ,G\right) $. And we
will denote by $Z\left( G\right) $ the center of the group $G$.

In our computation we will frequently use the identities%
\begin{equation}
(xy,z)=(x,z)^{y}(y,z)=(x,z)(x,z,y)(y,z),  \label{l_d}
\end{equation}%
\begin{equation}
(x,yz)=(x,z)(x,y)^{z}=(x,z)(x,y)(x,y,z),  \label{r_d}
\end{equation}%
\begin{equation}
(x^{-1},y)=(y,x)^{x^{-1}}=(x,y)^{-1}(y,x,x^{-1})  \label{i_d}
\end{equation}%
which fulfill in every group (see \cite[(10.2.1.2) and (10.2.1.3)]{Hall} and 
\cite[p. 20, (3)]{KM}). From these identities we can conclude these facts
about arbitrary group $G\in \mathfrak{N}_{4}$:

\begin{enumerate}
\item for every $g_{1},g_{2}\in G$ and every $l_{1},l_{2}\in \gamma
_{4}\left( G\right) $%
\begin{equation}
\left( g_{1}l_{1},g_{2}l_{2}\right) =\left( g_{1},g_{2}\right) ,
\label{g4out}
\end{equation}

\item for every $g\in G$ and every $l_{1},l_{2}\in \gamma _{2}\left(
G\right) $%
\begin{equation}
\left( l_{1}l_{2},g\right) =\left( l_{1},g\right) \left( l_{2},g\right) ,%
\hspace{0.2in}\left( g,l_{1}l_{2}\right) =\left( g,l_{2}\right) \left(
g,l_{1}\right) ,  \label{l_d_r_d}
\end{equation}

\item for every $g_{1},g_{2},g_{3}\in G$ and every $l_{1},l_{2},l_{3}\in
\gamma _{3}\left( G\right) $%
\begin{equation}
\left( g_{1}l_{1},g_{2}l_{2},g_{3}l_{3}\right) =\left(
g_{1},g_{2},g_{3}\right) ,  \label{g3out}
\end{equation}

\item for every $g_{1},g_{2},g_{3},g_{4}\in G$ and every $%
l_{1},l_{2},l_{3},l_{4}\in \gamma _{2}\left( G\right) $%
\begin{equation}
\left( g_{1}l_{1},g_{2}l_{2},g_{3}l_{3},g_{4}l_{4}\right) =\left(
g_{1},g_{2},g_{3},g_{4}\right) ,  \label{g2out}
\end{equation}

\item every commutator of the length $4$ is a multiplicative function by all
its $4$ arguments:%
\begin{equation}
w\left( g_{1},\ldots ,g_{i}l_{i},\ldots ,g_{4}\right) =w\left( g_{1},\ldots
,g_{i},\ldots ,g_{4}\right) w\left( g_{1},\ldots ,l_{i},\ldots ,g_{4}\right)
,  \label{g4_powers}
\end{equation}%
where $w\left( x_{1},\ldots ,x_{4}\right) \in \gamma _{4}\left( N_{4}\left(
x_{1},\ldots ,x_{4}\right) \right) $, $1\leq i\leq 4$, holds for every $%
g_{1},\ldots ,g_{4},l_{i}\in G$.
\end{enumerate}

For every $G\in \mathfrak{N}_{4}$ we have that $\gamma _{4}\left( G\right)
\subseteq Z\left( G\right) $ and $\gamma _{5}\left( G\right) =\left\{
1\right\} $. For every $G\in \Theta $ the group $\gamma _{2}\left( G\right) $
is an abelian group. We will use these facts later in our computations
without special reminder. Also we use the identity $yx=xy(y,x)$, which
fulfills in an every group, and the identity (\ref{exponent}) which fulfills
in an every group of the variety $\Theta $ without special reminder.

In this subsection we will describe the free group of our variety $\Theta $
generated by $2$ generators. This group is a quotient group $N_{4}\left(
x,y\right) /T$, where $T$ is a normal subgroup of the identities with two
variables of the subvariety $\Theta $ in the variety $\mathfrak{N}_{4}$.

By \cite[Theorem 17.2.2]{KM}, if $G$ is finitely generated nilpotent group
then there exist central (in particular, normal) series:

\begin{equation*}
G=G_{1}>G_{2}>...>G_{s}>G_{s+1}=\left\{ 1\right\}
\end{equation*}%
such that $G_{i}/G_{i+1}=\left\langle a_{i}G_{i+1}\right\rangle $ ($%
\Longleftrightarrow $ $G_{i}=\left\langle a_{i},G_{i+1}\right\rangle $), $%
a_{i}\in G_{i}$. \newline
$\left\langle a_{i}G_{i+1}\right\rangle \cong 
\mathbb{Z}
_{n}$ ($n\geq 2$), or $\left\langle a_{i}G_{i+1}\right\rangle \cong 
\mathbb{Z}
$. Therefore every $g\in G$ can be uniquely represented in the form $%
g=a_{1}^{\alpha _{1}}a_{2}^{\alpha _{2}}...a_{s}^{\alpha _{s}}$, where $%
0\leq \alpha _{i}<n$, when $\left\langle a_{i}G_{i+1}\right\rangle \cong 
\mathbb{Z}
_{n}$, and $\alpha _{i}\in 
\mathbb{Z}
$, when $\left\langle a_{i}G_{i+1}\right\rangle \cong 
\mathbb{Z}
$.

\begin{definition}
We say that the set $\left\{ a_{1},a_{2},...,a_{s}\right\} $ is \textbf{a
base} of the group $G$ and numbers $\alpha _{1},\alpha _{2},...,\alpha _{s}$
are \textbf{coordinates} of the element $g$ in this base.
\end{definition}

The base of $N_{4}\left( x,y\right) $ we can denote by%
\begin{equation}
C_{1}=x,C_{2}=y,C_{3}=(y,x),C_{4}=\left( y,x,y\right) ,C_{5}=\left(
y,x,x\right) ,  \label{baseN42}
\end{equation}%
\begin{equation*}
C_{6}=\left( y,x,x,x\right) ,C_{7}=(y,x,y,y),C_{8}=(y,x,y,x).
\end{equation*}%
This is a base of Shirshov, which we can compute by the algorithm explained
in \cite[2.3.5]{Bahturin}.

In particular, if we substitute in \cite[10.2.1.4]{Hall} $\left( y,x\right) $
instead $x$ and $x$ instead $z$, we obtain%
\begin{equation*}
\left( \left( y,x\right) ,y^{-1},x\right) ^{y}\left( y,x^{-1},\left(
y,x\right) \right) ^{x}\left( x,\left( y,x\right) ^{-1},y\right) ^{\left(
y,x\right) }=1.
\end{equation*}%
So, by (\ref{g4_powers}), we can conclude that%
\begin{equation*}
\left( \left( y,x\right) ,y,x\right) ^{-1}\left( y,x,\left( y,x\right)
\right) ^{-1}\left( x,\left( y,x\right) ,y\right) ^{-1}=1.
\end{equation*}%
We have that%
\begin{equation*}
\left( y,x,\left( y,x\right) \right) =\left( \left( y,x\right) ,\left(
y,x\right) \right) =1,
\end{equation*}%
and%
\begin{equation*}
\left( x,\left( y,x\right) ,y\right) =\left( \left( x,\left( y,x\right)
\right) ,y\right) =\left( \left( \left( y,x\right) ,x\right) ^{-1},y\right)
=\left( y,x,x,y\right) ^{-1},
\end{equation*}%
hence%
\begin{equation*}
\left( y,x,y,x\right) ^{-1}\left( y,x,x,y\right) =1
\end{equation*}%
and%
\begin{equation}
\left( y,x,y,x\right) =\left( y,x,x,y\right) =C_{8}.  \label{C8}
\end{equation}

\begin{proposition}
\label{collect_formula_applic}The identity%
\begin{equation}
(xy)^{4}=x^{4}y^{4}(y,x)^{6}(y,x,y)^{14}(y,x,y,y)^{11}(y,x,x)^{4}(y,x,x,y)^{11}(y,x,x,x),
\label{collect_formula}
\end{equation}%
fulfills in the variety $\mathfrak{N}_{4}$.
\end{proposition}

\begin{proof}
We will consider the group $G\in \mathfrak{N}_{4}$ and $x,y\in G$.

Initially we go to compute $(xy)^{2}$. We have that%
\begin{equation*}
(xy)^{2}=xyxy=x^{2}y(y,x)y=x^{2}y^{2}(y,x)(y,x,y)\text{.}
\end{equation*}

After this we compute $(xy)^{3}$ by same method:\newline
\begin{equation*}
(xy)^{3}=(xy)^{2}\left( xy\right) =x^{2}y^{2}(y,x)(y,x,y)xy=
\end{equation*}%
\begin{equation*}
x^{2}y^{2}x(y,x)(y,x,x)(y,x,y)(y,x,y,x)y\newline
.
\end{equation*}%
Now we will compute $y^{2}x$:%
\begin{equation*}
y^{2}x=y\left( yx\right) =yxy(y,x)=xy(y,x)y(y,x)=
\end{equation*}%
\begin{equation}
xy^{2}(y,x)(y,x,y)(y,x)=xy^{2}(y,x)^{2}(y,x,y),  \label{y^2x}
\end{equation}%
because elements of $\gamma _{2}\left( G\right) $ commute with elements of $%
\gamma _{3}\left( G\right) $ in every $G\in \mathfrak{N}_{4}$. Hence, by (%
\ref{y^2x}) and (\ref{C8}), we have the equality%
\begin{equation*}
(xy)^{3}=x^{3}y^{2}(y,x)^{3}(y,x,y)^{2}(y,x,x)(y,x,y,x)y=
\end{equation*}%
\begin{equation}
x^{3}y^{3}(y,x)^{3}(y,x,y)^{3}(y,x,y)^{2}(y,x,y,y)^{2}(y,x,x)(y,x,x.y)(y,x,y,x)=
\label{x^3y^3}
\end{equation}%
\begin{equation*}
x^{3}y^{3}(y,x)^{3}(y,x,y)^{5}(y,x,x)(y,x,y,y)^{2}(y,x,x,y)^{2}.
\end{equation*}

Now we will compute $(xy)^{4}$. By (\ref{x^3y^3}) we have that%
\begin{equation}
(xy)^{4}=xy(xy)^{3}=xyx^{3}y^{3}(y,x)^{3}(y,x,y)^{5}(y,x,x)(y,x,y,y)^{2}(y,x,y,x)^{2}
\label{(xy)4}
\end{equation}%
After this we $\bigskip $can compute that%
\begin{equation*}
yx^{3}=\left( yx\right) x^{2}=xy(y,x)x^{2}=x^{2}y(y,x)^{2}(y,x,x)x=
\end{equation*}%
\begin{equation*}
x^{3}y\left( y,x\right) ^{3}\left( y,x,x\right) ^{3}\left( y,x,x,x\right) ,
\end{equation*}%
therefore%
\begin{equation}
xyx^{3}y^{3}=x^{4}y\left( y,x\right) ^{3}\left( y,x,x\right) ^{3}\left(
y,x,x,x\right) y^{3}.  \label{xyx^3y^3_1}
\end{equation}%
We have that%
\begin{equation}
\left( y,x,x,x\right) y^{3}=y^{3}\left( y,x,x,x\right) .  \label{(y,x,x,x)y3}
\end{equation}%
Also we can compute that 
\begin{equation*}
\left( y,x,x\right) ^{3}y^{3}=y\left( y,x,x\right) ^{3}\left( y,x,x,y\right)
^{3}y^{2}=y\left( y,x,x\right) ^{3}y^{2}\left( y,x,x,y\right) ^{3}=
\end{equation*}%
\begin{equation}
y^{2}\left( y,x,x\right) ^{3}y\left( y,x,x,y\right) ^{6}=y^{3}\left(
y,x,x\right) ^{3}\left( y,x,x,y\right) ^{9}  \label{(y,x,x)3y3}
\end{equation}%
and%
\begin{equation*}
\left( y,x\right) ^{3}y^{3}=y\left( y,x\right) ^{3}\left( y,x,y\right)
^{3}y^{2}=
\end{equation*}%
\begin{equation*}
y^{2}\left( y,x\right) ^{3}\left( y,x,y\right) ^{3}\left( y,x,y\right)
^{3}\left( y,x,y,y\right) ^{3}y=
\end{equation*}%
\begin{equation}
y^{2}\left( y,x\right) ^{3}\left( y,x,y\right) ^{6}\left( y,x,y,y\right)
^{3}y=  \label{(y,x)3y3}
\end{equation}%
\begin{equation*}
y^{3}\left( y,x\right) ^{3}\left( y,x.y\right) ^{3}\left( y,x,y\right)
^{6}\left( y,x,y,y\right) ^{6}\left( y,x,y,y\right) ^{3}=
\end{equation*}%
\begin{equation*}
y^{3}\left( y,x\right) ^{3}\left( y,x.y\right) ^{9}\left( y,x,y,y\right)
^{9}.
\end{equation*}%
Therefore, by (\ref{xyx^3y^3_1}), (\ref{(y,x,x,x)y3}), (\ref{(y,x,x)3y3})
and (\ref{(y,x)3y3}), 
\begin{equation}
xyx^{3}y^{3}=x^{4}y^{4}\left( y,x\right) ^{3}\left( y,x.y\right) ^{9}\left(
y,x,y,y\right) ^{9}\left( y,x,x\right) ^{3}\left( y,x,x,y\right) ^{9}\left(
y,x,x,x\right) .  \label{xyx3y3}
\end{equation}%
After this, we have, by (\ref{(xy)4}) and (\ref{xyx3y3}), that%
\begin{equation*}
(xy)^{4}=x^{4}y^{4}\left( y,x\right) ^{3}\left( y,x.y\right) ^{9}\left(
y,x,y,y\right) ^{9}\left( y,x,x\right) ^{3}\left( y,x,x,y\right) ^{9}\left(
y,x,x,x\right) \cdot
\end{equation*}%
\begin{equation*}
(y,x)^{3}(y,x,y)^{5}(y,x,x)(y,x,y,y)^{2}(y,x,y,x)^{2}=
\end{equation*}%
\begin{equation*}
x^{4}y^{4}\left( y,x\right) ^{6}\left( y,x.y\right) ^{14}\left(
y,x,y,y\right) ^{11}\left( y,x,x\right) ^{4}\left( y,x,x,y\right)
^{11}\left( y,x,x,x\right) .
\end{equation*}
\end{proof}

By (\ref{C8}) we have the

\begin{corollary}
The identity 
\begin{equation}
1=(y,x)^{2}(y,x,y)^{2}(y,x,x,x)(y,x,y,y)^{-1}(y,x,y,x)^{-1}.
\label{collectionFormula}
\end{equation}%
fulfills in the variety $\Theta $.
\end{corollary}

\setcounter{corollary}{0}

We denote the images of elements of the base $\left\{ C_{1},\ldots
,C_{8}\right\} $ by the natural homomorphism $N_{4}\left( x,y\right)
\rightarrow N_{4}\left( x,y\right) /T=F_{\Theta }\left( x,y\right) $ by same
notation: $\left\{ C_{1},\ldots ,C_{8}\right\} $.

\begin{proposition}
\label{relations}The relations:%
\begin{equation}
C_{i}^{2}=1,(4\leq i\leq 8)  \label{Rxy1}
\end{equation}%
\begin{equation}
C_{3}^{2}C_{6}C_{7}C_{8}=1  \label{Rxy2}
\end{equation}%
in $F_{\Theta }\left( x,y\right) $ are conclusions from the identities of $%
\Theta $.
\end{proposition}

\begin{proof}
The (\ref{collectionFormula}) is an identity in $\Theta $, so in (\ref%
{collectionFormula}) we can substitute $x$ instead $y$ and vice versa.
Therefore%
\begin{equation*}
1=(x,y)^{2}(x,y,x)^{2}(x,y,y,y)(x,y,x,x)^{-1}(x,y,x,y)^{-1}=
\end{equation*}%
\begin{equation}
(y,x)^{2}(y,x,y)^{2}(y,x,x,x)(y,x,y,y)^{-1}(y,x,y,x)^{-1}.
\label{collectionFormula_x_y_y_x}
\end{equation}%
\newline
By (\ref{i_d}), (\ref{g2out}), (\ref{g4_powers}) and (\ref{C8}) we have that 
$(y,x)^{2}=(x,y)^{2}$, $(y,x,x,x)=(x,y,x,x)^{-1}$, $(x,y,y,y)=(y,x,y,y)^{-1}$%
, $(y,x,y,x)=(x,y,x,y)^{-1}$. Therefore we conclude from (\ref%
{collectionFormula_x_y_y_x}) that%
\begin{equation}
(x,y,x)^{2}(y,x,y,x)=(y,x,y)^{2}(y,x,y,x)^{-1}.
\label{collectionFormula_x_y_y_x_2}
\end{equation}%
Also we have by (\ref{i_d}), that%
\begin{equation*}
(x,y,x)=(y,x,x)^{-1}\left( x,(y,x),(x,y)\right) =(y,x,x)^{-1}=C_{5}^{-1}.
\end{equation*}%
Therefore $(x,y,x)^{2}=C_{5}^{-2}=C_{5}^{2}$. Now we conclude from (\ref%
{collectionFormula_x_y_y_x_2}) that%
\begin{equation}
C_{4}^{2}=C_{5}^{2}C_{8}^{2}\newline
.  \label{2_4__2_52_8}
\end{equation}

Now we substitute in (\ref{collectionFormula}) $(y,x)$ instead $x$ and $x$
instead $y$:\qquad 
\begin{equation*}
1=(x,(y,x))^{2}(x,(y,x),x)^{2}(x,(y,x),(y,x),(y,x))\cdot
\end{equation*}%
\begin{equation*}
(x,(y,x),x,x)^{-1}(x,(y,x),x,(y,x))^{-1}=
\end{equation*}%
\begin{equation*}
(x,(y,x))^{2}(x,(y,x),x)^{2}=(y,x,x)^{-2}(y,x,x,x)^{-2}
\end{equation*}%
So the relation%
\begin{equation}
C_{5}^{2}C_{6}^{2}=1  \label{2_52_6}
\end{equation}%
holds.

Analogously we substitute in (\ref{collectionFormula}) $y$ instead $x$ and $%
(y,x)$ instead $y$ and conclude that%
\begin{equation}
1=C_{4}^{2}.  \label{2_4}
\end{equation}%
Now by (\ref{2_4__2_52_8}) and (\ref{2_52_6}) we have that%
\begin{equation}
C_{5}^{2}=C_{6}^{2}=C_{8}^{2}.  \label{2_5__2_6__2_8}
\end{equation}

Also, when we substitute in (\ref{collectionFormula}) $(y,x,x)$ instead $y$,
we obtain that 
\begin{equation}
1=C_{6}^{2}.  \label{2_6}
\end{equation}

And when we substitute in (\ref{collectionFormula}) $(y,x,y)$ instead $x$,
we conclude%
\begin{equation}
1=C_{7}^{-2}=C_{7}^{,2}.  \label{2_7}
\end{equation}%
Therefore, we conclude (\ref{Rxy1}) from (\ref{2_4}), (\ref{2_5__2_6__2_8}),
(\ref{2_6}), (\ref{2_7}). And after this the (\ref{collectionFormula}) has
form%
\begin{equation*}
1=C_{3}^{2}C_{4}^{2}C_{6}C_{7}^{-1}C_{8}^{-1}=C_{3}^{2}C_{6}C_{7}C_{8}.
\end{equation*}
\end{proof}

Now we consider in the group $N_{4}\left( x,y\right) $ the minimal normal
subgroup $R$ which contains elements $x^{4}$, $y^{4}$ and the left parts of
the relations (\ref{Rxy1}) and (\ref{Rxy2}). Here we consider the elements $%
x=C_{1},y=C_{2}$, and $C_{3},\ldots ,C_{8}$ as elements of $N_{4}\left(
x,y\right) $. The images of the elements $C_{1},\ldots ,C_{8}$ by the
natural epimorphism $N_{4}\left( x,y\right) \rightarrow N_{4}\left(
x,y\right) /R$ we also denote by $C_{1},\ldots ,C_{8}$. We see from the
Proposition \ref{relations} that the base of the group $N_{4}\left(
x,y\right) /R$ is $\left\{ C_{1},C_{2},\ldots ,C_{7}\right\} $ and, if $%
1\leq i\leq 3$, then $\left\vert C_{i}\right\vert =4$, if $4\leq i\leq 7$,
then $\left\vert C_{i}\right\vert =2$.

Our goal is to prove that $N_{4}\left( x,y\right) /R=F_{\Theta }\left(
x,y\right) $. For this we must study the group $N_{4}\left( x,y\right) /R$
and prove some lemmas about it's properties. These lemmas we will use in the
proof of the Theorem \ref{freeGroup} and in other computations.

\section{Some lemmas about the group $N_{4}\left( x,y\right) /R$}

\setcounter{equation}{0}

In this section we will denote the group $N_{4}\left( x,y\right) /R$ by $G$.

\begin{lemma}
\label{g2}$\gamma _{2}\left( G\right) $ is a commutative group.
\end{lemma}

\begin{proof}
We have that $\gamma _{3}\left( N_{4}\left( x,y\right) \right) \leq Z\left(
\gamma _{2}\left( N_{4}\left( x,y\right) \right) \right) $ and quotient
group $\gamma _{2}\left( N_{4}\left( x,y\right) \right) /\gamma _{3}\left(
N_{4}\left( x,y\right) \right) =\left\langle \left( y,x\right) \gamma
_{3}\left( N_{4}\left( x,y\right) \right) \right\rangle $ is a cyclic group.
Therefore $\gamma _{2}\left( N_{4}\left( x,y\right) \right) $ is a
commutative group. $G$ is a homomorphic image of $N_{4}\left( x,y\right) $,
so, $\gamma _{2}\left( G\right) $ is a commutative group.
\end{proof}

\begin{lemma}
\label{g3}The group $\gamma _{3}\left( G\right) $ is a group of exponent $2$.
\end{lemma}

\begin{proof}
We have that $\gamma _{3}\left( G\right) =\left\langle C_{4},\ldots
,C_{7}\right\rangle $. Lemma \ref{g2} and the consideration of relations (%
\ref{Rxy1}) completes the proof.
\end{proof}

\begin{lemma}
\label{C3_2}For every $h\in \gamma _{2}\left( G\right) $ the inclusion $%
h^{2}\in \gamma _{4}\left( G\right) $ holds.
\end{lemma}

\begin{proof}
We have that $\gamma _{2}\left( G\right) =\left\langle C_{3},\ldots
,C_{7}\right\rangle $. Lemma \ref{g2} and the consideration of relations (%
\ref{Rxy1}) and (\ref{Rxy2}) completes the proof.
\end{proof}

\begin{lemma}
\label{q_d}For every $a,b,c\in G$ the following equalities holds:%
\begin{equation}
(ab,c)^{2}=(a,c)^{2}(b,c)^{2},  \label{q_l_d}
\end{equation}%
\begin{equation}
(a,bc)^{2}=(a,c)^{2}(a,b)^{2}  \label{q_r_d}
\end{equation}%
\begin{equation}
(a^{-1},b)^{2}=(a,b)^{2}  \label{q_i_l}
\end{equation}%
\begin{equation}
(a,b^{-1})^{2}=(a,b)^{2}.  \label{q_i_r}
\end{equation}
\end{lemma}

\begin{proof}
We have that%
\begin{equation*}
(ab,c)^{2}=\left( (a,c)^{b}(b,c)\right) ^{2}=\left( (a,c)^{2}\right)
^{b}(b,c)^{2}
\end{equation*}%
by (\ref{l_d}) and by Lemma \ref{g2}. And now\ by Lemma \ref{C3_2} we
conclude (\ref{q_l_d}). By similar computation we can conclude (\ref{q_l_d})
from (\ref{r_d}) and Lemma \ref{C3_2}.

By Lemmas \ref{C3_2}\ and \ref{g3} $\gamma _{2}\left( G\right) $ is a group
of exponent $4$. Therefore, by (\ref{i_d}) and Lemma \ref{C3_2} we have that%
\begin{equation*}
(a^{-1},b)^{2}=\left( (b,a)^{a^{-1}}\right) ^{2}=\left( (b,a)^{2}\right)
^{a^{-1}}=(b,a)^{2}=(a,b)^{-2}=(a,b)^{2}.
\end{equation*}%
By similar computation we can conclude (\ref{q_i_r}).
\end{proof}

\begin{lemma}
\label{gh}If $g\in G$, $h\in \gamma _{2}\left( G\right) $, then the $\left(
gh\right) ^{4}=g^{4}$.
\end{lemma}

\begin{proof}
We know that the identity (\ref{collect_formula}) holds in the variety $%
\mathfrak{N}_{4}$. So this identity holds in $G$. Hence we have that%
\begin{equation*}
(gh)^{4}=g^{4}h^{4}(h,g)^{6}(h,g,h)^{14}(h,g,h,h)^{11}(h,g,g)^{4}(h,g,g,h)^{11}(h,g,g,g).
\end{equation*}%
In our case $(h,g,h),(h,g,h,h),(h,g,g,h),(h,g,g,g)\in \gamma _{5}\left(
G\right) $. By Lemmas \ref{C3_2}\ and \ref{g3} we have that $%
h^{4}=(h,g)^{6}=(h,g,g)^{4}=1$. Therefore $(gh)^{4}=g^{4}$.
\end{proof}

\section{Computation of the group $F_{\Theta }\left( x,y\right) $}

\setcounter{equation}{0}

\begin{theorem}
\label{freeGroup}$N_{4}\left( x,y\right) /R=F_{\Theta }\left( x,y\right) $.
\end{theorem}

\begin{proof}
In this proof we also denote the group $N_{4}\left( x,y\right) /R$ by $G$.

By Proposition \ref{relations} the relations $r=1$, where $r\in R$, are
conclusions from the identities which define the variety $\Theta $. So we
only must prove that $G\in \Theta $.

It is clear that the group $G$ is a nilpotent groups of class $4$.

As we said in the proof of the Lemma \ref{C3_2}, $G$ is a metabelian group.

Now we will prove that the group $G$ fulfills the identity (\ref{exponent}).
By Lemma \ref{gh}, it remains for us to prove now that for every $0\leq
\alpha _{1},\alpha _{2}\leq 3$ the%
\begin{equation*}
\left( x^{\alpha _{1}}y^{\alpha _{2}}\right) ^{4}=1
\end{equation*}%
holds in $G$. We substitute in (\ref{collect_formula}) $x^{\alpha _{1}}$
instead $x$ and $y^{\alpha _{2}}$ instead $y$. $\left( x^{\alpha
_{1}}\right) ^{4}=\left( y^{\alpha _{2}}\right) ^{4}=1$ holds in $G$.
Therefore we must only prove that%
\begin{equation*}
(y^{\alpha _{2}},x^{\alpha _{1}})^{6}(y^{\alpha _{2}},x^{\alpha
_{1}},y^{\alpha _{2}})^{14}(y^{\alpha _{2}},x^{\alpha _{1}},y^{\alpha
_{2}},y^{\alpha _{2}})^{11}(y^{\alpha _{2}},x^{\alpha _{1}},x^{\alpha
_{1}})^{4}\cdot
\end{equation*}%
\begin{equation*}
(y^{\alpha _{2}},x^{\alpha _{1}},x^{\alpha _{1}},y^{\alpha
_{2}})^{11}(y^{\alpha _{2}},x^{\alpha _{1}},x^{\alpha _{1}},x^{\alpha
_{1}})=1
\end{equation*}%
holds in $G$. By Lemmas \ref{C3_2}\ and \ref{g3} we have that%
\begin{equation*}
(y^{\alpha _{2}},x^{\alpha _{1}})^{6}=(y^{\alpha _{2}},x^{\alpha _{1}})^{2},
\end{equation*}%
\begin{equation*}
(y^{\alpha _{2}},x^{\alpha _{1}},y^{\alpha _{2}})^{14}=1,
\end{equation*}%
\begin{equation*}
(y^{\alpha _{2}},x^{\alpha _{1}},y^{\alpha _{2}},y^{\alpha
_{2}})^{11}=(y^{\alpha _{2}},x^{\alpha _{1}},y^{\alpha _{2}},y^{\alpha
_{2}}),
\end{equation*}%
\begin{equation*}
(y^{\alpha _{2}},x^{\alpha _{1}},x^{\alpha _{1}})^{4}=1,
\end{equation*}%
\begin{equation*}
(y^{\alpha _{2}},x^{\alpha _{1}},x^{\alpha _{1}},y^{\alpha
_{2}})^{11}=(y^{\alpha _{2}},x^{\alpha _{1}},x^{\alpha _{1}},y^{\alpha
_{2}}).
\end{equation*}%
We denote%
\begin{equation*}
v\left( \alpha _{1},\alpha _{2}\right) =(y^{\alpha _{2}},x^{\alpha
_{1}})^{2}(y^{\alpha _{2}},x^{\alpha _{1}},y^{\alpha _{2}},y^{\alpha
_{2}})(y^{\alpha _{2}},x^{\alpha _{1}},x^{\alpha _{1}},y^{\alpha
_{2}})(y^{\alpha _{2}},x^{\alpha _{1}},x^{\alpha _{1}},x^{\alpha _{1}})
\end{equation*}%
So it remains for us to prove that%
\begin{equation}
v\left( \alpha _{1},\alpha _{2}\right) =1  \label{ord4}
\end{equation}%
holds in $G$ for every $0\leq \alpha _{1},\alpha _{2}\leq 3$.

It is clear that (\ref{ord4}) holds in $G$ if $\alpha _{1}=0$ or $\alpha
_{2}=0$. If $\alpha _{1}=\alpha _{2}=1$, than by (\ref{collect_formula}), we
have that%
\begin{equation*}
v\left( 1,1\right)
=(y,x)^{2}(y,x,y,y)(y,x,x,y)(y,x,x,x)=C_{3}^{2}C_{7}C_{8}C_{6}=1.
\end{equation*}%
Now we will prove (\ref{ord4}) by induction on $\alpha _{1}$ and $\alpha
_{2} $. We suppose that (\ref{ord4}) holds for all $\alpha _{1},\alpha _{2}$
such that $\alpha _{1}\leq \beta _{1}$, $\alpha _{2}\leq \beta _{2}$ and $%
0\leq \alpha _{1},\alpha _{2}$. We have by (\ref{q_r_d}) that%
\begin{equation}
(y^{\beta _{2}},x^{\beta _{1}+1})^{2}=\left( y^{\beta _{2}},x\right)
^{2}\left( y^{\beta _{2}},x^{\beta _{1}}\right) ^{2}.  \label{x_q}
\end{equation}%
We have by (\ref{g4_powers}) that%
\begin{equation}
(y^{\beta _{2}},x^{\beta _{1}+1},y^{\beta _{2}},y^{\beta _{2}})=(y^{\beta
_{2}},x^{\beta _{1}},y^{\beta _{2}},y^{\beta _{2}})(y^{\beta
_{2}},x,y^{\beta _{2}},y^{\beta _{2}}),  \label{1x}
\end{equation}%
\begin{equation*}
(y^{\beta _{2}},x^{\beta _{1}+1},x^{\beta _{1}+1},y^{\beta _{2}})=(y^{\beta
_{2}},x,x,y^{\beta _{2}})^{\left( \beta _{1}+1\right) ^{2}}=
\end{equation*}%
\begin{equation*}
(y^{\beta _{2}},x,x,y^{\beta _{2}})^{\beta _{1}^{2}}(y^{\beta
_{2}},x,x,y^{\beta _{2}})^{2\beta _{1}}(y^{\beta _{2}},x,x,y^{\beta _{2}})
\end{equation*}%
and%
\begin{equation*}
(y^{\beta _{2}},x^{\beta _{1}+1},x^{\beta _{1}+1},x^{\beta
_{1}+1})=(y^{\beta _{2}},x,x,x)^{\left( \beta _{1}+1\right) ^{3}}=
\end{equation*}%
\begin{equation*}
(y^{\beta _{2}},x,x,x)^{\beta _{1}^{3}}(y^{\beta _{2}},x,x,x)^{3\beta
_{1}\left( \beta _{1}+1\right) }(y^{\beta _{2}},x,x,x).
\end{equation*}%
By Lemma \ref{g3} we have that%
\begin{equation*}
(y^{\beta _{2}},x,x,y^{\beta _{2}})^{2\beta _{1}}=(y^{\beta
_{2}},x,x,x)^{3\beta _{1}\left( \beta _{1}+1\right) }=1,
\end{equation*}%
because $\beta _{1}\left( \beta _{1}+1\right) $ is an even number. Hence%
\begin{equation}
(y^{\beta _{2}},x^{\beta _{1}+1},x^{\beta _{1}+1},y^{\beta _{2}})=(y^{\beta
_{2}},x,x,y^{\beta _{2}})^{\beta _{1}^{2}}(y^{\beta _{2}},x,x,y^{\beta
_{2}})=  \label{2x}
\end{equation}%
\begin{equation*}
(y^{\beta _{2}},x^{\beta _{1}},x^{\beta _{1}},y^{\beta _{2}})(y^{\beta
_{2}},x,x,y^{\beta _{2}})
\end{equation*}%
and%
\begin{equation}
(y^{\beta _{2}},x^{\beta _{1}+1},x^{\beta _{1}+1},x^{\beta
_{1}+1})=(y^{\beta _{2}},x,x,x)^{\beta _{1}^{3}}(y^{\beta _{2}},x,x,x)=
\label{3x}
\end{equation}%
\begin{equation*}
(y^{\beta _{2}},x^{\beta _{1}},x^{\beta _{1}},x^{\beta _{1}})(y^{\beta
_{2}},x,x,x).
\end{equation*}%
Therefore, by (\ref{x_q}), (\ref{1x}), (\ref{2x}), (\ref{3x}) and by our
hypothesis about $v\left( \beta _{1},\beta _{2}\right) $ and $v\left(
1,\beta _{2}\right) $, we have that 
\begin{equation*}
v\left( \beta _{1}+1,\beta _{2}\right) =v\left( \beta _{1},\beta _{2}\right)
v\left( 1,\beta _{2}\right) =1.
\end{equation*}

By (\ref{q_l_d}) we have that%
\begin{equation}
\left( y^{\beta _{2}+1},x^{\beta _{1}}\right) ^{2}=\left( y^{\beta
_{2}},x^{\beta _{1}}\right) ^{2}\left( y,x^{\beta _{1}}\right) ^{2},
\label{y_q}
\end{equation}%
And now, similarly to the previous arguments, we conclude that%
\begin{equation}
(y^{\beta _{2}+1},x^{\beta _{1}},x^{\beta _{1}},x^{\beta _{1}})=(y^{\beta
_{2}},x^{\beta _{1}},x^{\beta _{1}},x^{\beta _{1}})(y,x^{\beta
_{1}},x^{\beta _{1}},x^{\beta _{1}}),  \label{1y}
\end{equation}%
\begin{equation}
(y^{\beta _{2}+1},x^{\beta _{1}},x^{\beta _{1}},y^{\beta _{2}+1})=(y^{\beta
_{2}},x^{\beta _{1}},x^{\beta _{1}},y^{\beta _{2}})(y,x^{\beta
_{1}},x^{\beta _{1}},y),  \label{2y}
\end{equation}%
and%
\begin{equation}
(y^{\beta _{2}+1},x^{\beta _{1}},y^{\beta _{2}+1},y^{\beta
_{2}+1})=(y^{\beta _{2}},x^{\beta _{1}},y^{\beta _{2}},y^{\beta
_{2}})(y,x^{\beta _{1}},y,y).  \label{3y}
\end{equation}%
Hence, by (\ref{y_q}), (\ref{1y}), (\ref{2y}), (\ref{3y}) and by our
hypothesis about $v\left( \beta _{1},\beta _{2}\right) $ and $v\left( \beta
_{1},1\right) $,%
\begin{equation*}
v\left( \beta _{1},\beta _{2}+1\right) =v\left( \beta _{1},\beta _{2}\right)
v\left( \beta _{1},1\right) =1.
\end{equation*}%
Therefore we prove that (\ref{ord4}) holds in $G$ for every $0\leq \alpha
_{1},\alpha _{2}\leq 3$. This completes our proof.
\end{proof}

Now, when we know that $N_{4}\left( x,y\right) /R=F_{\Theta }\left(
x,y\right) $, we can prove the

\begin{corollary}
\label{theta}The Lemmas \ref{g2}, \ref{g3}, \ref{C3_2} and \ref{q_d} hold
when we consider as group $G$ an arbitrary group of the variety $\Theta $.
\end{corollary}

\begin{proof}
Lemma \ref{g2} is fulfilled by definition of the variety $\Theta $.

Now we will prove that every $G\in \Theta $ fulfills the conclusion of the
Lemma \ref{C3_2}. Every $h\in \gamma _{2}\left( G\right) $ is generated by
commutators $\left( a,b\right) $, such that $a,b\in G$. There exists the
homomorphism $\varphi :F_{\Theta }\left( x,y\right) \rightarrow G$ such that 
$\varphi \left( x\right) =b$, $\varphi \left( y\right) =a$. We apply $%
\varphi $ to (\ref{Rxy2}) and conclude that $\left( a,b\right) ^{2}\in
\gamma _{4}\left( G\right) $.

Also every $G\in \Theta $ fulfills the conclusion of the Lemma \ref{g3},
because the group $\gamma _{3}\left( G\right) $ is generated by the
commutators $\left( a,b\right) $, where $a\in G$, $b\in \gamma _{2}\left(
G\right) $. We also consider the homomorphism $\varphi :F_{\Theta }\left(
x,y\right) \rightarrow G$ from the previous part of the proof, apply it to (%
\ref{Rxy2}) and now, because $b\in \gamma _{2}\left( G\right) $, conclude
that $\left( a,b\right) ^{2}\in \gamma _{5}\left( G\right) $.

The proof of the fact that every $G\in \Theta $ fulfills the conclusion of
the Lemma \ref{q_d} coincides with the proof of the Lemma \ref{q_d} for the
group $N_{4}\left( x,y\right) /R$.
\end{proof}

\setcounter{corollary}{0}

\section{Applicable systems of words. Necessary conditions}

\setcounter{equation}{0}

\begin{proposition}
\label{AN1}If $W$ (see (\ref{syst_words}) ) is applicable system of words in
our variety $\Theta $ then always $w_{1}=1$, $w_{-1}\left( x\right) =x^{-1}$.
\end{proposition}

\begin{proof}
We suppose that $W$ is an applicable system of words. $w_{1}\in F_{\Theta
}\left( \varnothing \right) =\left\{ 1\right\} $, so $w_{1}=1$.

$w_{-1}\left( x\right) \in F_{\Theta }\left( x\right) \cong 
\mathbb{Z}
_{4}$. We denote $F_{\Theta }\left( x\right) $ by $F$. Because $W$ is
applicable system of words, by Definition \ref{asw}, there exists the
isomorphism $s_{F}:F\rightarrow F_{W}^{\ast }$, such that $s_{F}\left(
x\right) =x$. We have that $s_{F}\left( x^{-1}\right) =w_{-1}\left(
s_{F}\left( x\right) \right) =w_{-1}\left( x\right) $. If $w_{-1}\left(
x\right) =1$, then $s_{F}\left( x^{-1}\right) =1$, but $s_{F}\left( 1\right)
=w_{1}=1$, but this contradicts the assumption that $s_{F}$ is an injective
mapping.

If $w_{-1}\left( x\right) =x$, then $s_{F}\left( x^{-1}\right)
=x=s_{F}\left( x\right) $ which gives the same contradiction.

If $w_{-1}\left( x\right) =x^{2}$, then $x=s_{F}\left( x\right) =s_{F}\left(
\left( x^{-1}\right) ^{-1}\right) =w_{-1}\left( w_{-1}\left( x\right)
\right) =w_{-1}\left( x^{2}\right) =\left( x^{2}\right) ^{2}=x^{4}=1$. It
also gives us a contradiction.

Therefore, there is only one possibility: $w_{-1}\left( x\right)
=x^{3}=x^{-1}$.
\end{proof}

For studying of words $w_{\cdot }\left( x,y\right) $ we need to consider the
group $F_{\Theta }\left( x,y\right) $. We denote this group by $G$.

Because $W$ is applicable system of words, by Definition \ref{asw}, there
exists the isomorphism $s_{G}:G\rightarrow G_{W}^{\ast }$, which fix $x$ and 
$y$.

\begin{proposition}
\label{AN2_PR}If $W$ (see (\ref{syst_words}) ) is applicable system of words
in our variety $\Theta $ then always%
\begin{equation}
w_{\cdot }\left( x,y\right) =xyC_{3}^{\alpha _{3}}C_{4}^{\alpha
_{4}}C_{5}^{\alpha _{5}}C_{6}^{\alpha _{6}}C_{7}^{\alpha _{7}},  \label{w}
\end{equation}%
where $0\leq \alpha _{3}<4$, $\alpha _{i}\in \left\{ 0,1\right\} $, when $%
4\leq i\leq 7$.
\end{proposition}

\begin{proof}
We use the considerations of \cite[Proposition 2.1]{TsurkovNilpotent}. $%
w_{\cdot }\left( x,y\right) \in G$, so $w_{\cdot }\left( x,y\right)
=x^{\alpha _{1}}y^{\alpha _{2}}g_{2}\left( x,y\right) $, where $g_{2}\left(
x,y\right) \in \gamma _{2}\left( G\right) $, $0\leq \alpha _{1},\alpha
_{2}<4 $.

We have that $x=s_{G}\left( x\cdot 1\right) =w_{\cdot }\left( s_{G}\left(
x\right) ,s_{G}\left( 1\right) \right) =w_{\cdot }\left( x,w_{1}\right)
=w_{\cdot }\left( x,1\right) =x^{\alpha _{1}}g_{2}\left( x,1\right)
=x^{\alpha _{1}}$ holds, because $g_{2}\left( x,1\right) $ is the result of
substitution of $1$ instead $y$ in $g_{2}\left( x,y\right) $. Therefore, $%
\alpha _{1}=1$. We obtain by similar computations that $\alpha _{2}=1$.
\end{proof}

In the next Proposition we will get the stronger result about word $w_{\cdot
}\left( x,y\right) $ from applicable system of words $W$.

\begin{proposition}
\label{AN2}If $W$ (see (\ref{syst_words}) ) is applicable system of words in
our variety $\Theta $ then always $w_{\cdot }\left( x,y\right)
=xyC_{3}^{\alpha _{3}}$, where $\alpha _{3}=0,1,2,3$.
\end{proposition}

\begin{proof}
The equalities $x\left( xy\right) =\left( xx\right) y$ and $x\left(
yy\right) =\left( xy\right) y$ hold in $G=F_{\Theta }\left( x,y\right) $. We
apply the isomorphism $s_{G}:G\rightarrow G_{W}^{\ast }$ to the both hands
of the first equality and have that%
\begin{equation*}
s_{G}\left( x\left( xy\right) \right) =w_{\cdot }\left( s_{G}\left( x\right)
,s_{G}\left( xy\right) \right) =
\end{equation*}%
\begin{equation*}
w_{\cdot }\left( s_{G}\left( x\right) ,w_{\cdot }\left( s_{G}\left( x\right)
,s_{G}\left( y\right) \right) \right) =w_{\cdot }\left( x,w_{\cdot }\left(
x,y\right) \right)
\end{equation*}%
and%
\begin{equation*}
s_{G}\left( \left( xx\right) y\right) =w_{\cdot }\left( s_{G}\left(
xx\right) ,s_{G}\left( y\right) \right) =
\end{equation*}%
\begin{equation*}
w_{\cdot }\left( w_{\cdot }\left( s_{G}\left( x\right) ,s_{G}\left( x\right)
\right) ,s_{G}\left( y\right) \right) =w_{\cdot }\left( w_{\cdot }\left(
x,x\right) ,y\right) .
\end{equation*}%
Therefore%
\begin{equation*}
w_{\cdot }\left( x,w_{\cdot }\left( x,y\right) \right) =w_{\cdot }\left(
w_{\cdot }\left( x,x\right) ,y\right) .
\end{equation*}%
We conclude be similar computations from the second equality that%
\begin{equation*}
w_{\cdot }\left( x,w_{\cdot }\left( y,y\right) \right) =w_{\cdot }\left(
w_{\cdot }\left( x,y\right) ,y\right)
\end{equation*}%
holds. If we will denote the operation defined by the word (\ref{w}) by
symbol $\circ $, then we can rewrite these equalities in the form 
\begin{equation}
x\circ \left( x\circ y\right) =\left( x\circ x\right) \circ y  \label{2X}
\end{equation}%
and%
\begin{equation}
x\circ \left( y\circ y\right) =\left( x\circ y\right) \circ y.  \label{2Y}
\end{equation}

Now we will compute the left hand of (\ref{2X}). We have that%
\begin{equation*}
x\circ \left( x\circ y\right) =x\circ xyC_{3}^{\alpha _{3}}C_{4}^{\alpha
_{4}}C_{5}^{\alpha _{5}}C_{6}^{\alpha _{6}}C_{7}^{\alpha _{7}}=
\end{equation*}%
\begin{equation}
xxyC_{3}^{\alpha _{3}}C_{4}^{\alpha _{4}}C_{5}^{\alpha _{5}}C_{6}^{\alpha
_{6}}C_{7}^{\alpha _{7}}\cdot L_{3}^{\alpha _{3}}L_{4}^{\alpha
_{4}}L_{5}^{\alpha _{5}}L_{6}^{\alpha _{6}}L_{7}^{\alpha _{7}},
\label{xc(xcy)}
\end{equation}%
where%
\begin{equation}
L_{3}=(xyC_{3}^{\alpha _{3}}C_{4}^{\alpha _{4}}C_{5}^{\alpha
_{5}}C_{6}^{\alpha _{6}}C_{7}^{\alpha _{7}},x)=(xyC_{3}^{\alpha
_{3}}C_{4}^{\alpha _{4}}C_{5}^{\alpha _{5}},x)  \label{L3_def}
\end{equation}%
by (\ref{g4out}),%
\begin{equation}
L_{4}=\left( xyC_{3}^{\alpha _{3}}C_{4}^{\alpha _{4}}C_{5}^{\alpha
_{5}}C_{6}^{\alpha _{6}}C_{7}^{\alpha _{7}},x,xyC_{3}^{\alpha
_{3}}C_{4}^{\alpha _{4}}C_{5}^{\alpha _{5}}C_{6}^{\alpha _{6}}C_{7}^{\alpha
_{7}}\right) =\left( xyC_{3}^{\alpha _{3}},x,xyC_{3}^{\alpha _{3}}\right)
\label{L4_def}
\end{equation}%
by (\ref{g3out}),%
\begin{equation}
L_{5}=\left( xyC_{3}^{\alpha _{3}}C_{4}^{\alpha _{4}}C_{5}^{\alpha
_{5}}C_{6}^{\alpha _{6}}C_{7}^{\alpha _{7}},x,x\right) =\left(
xyC_{3}^{\alpha _{3}},x,x\right)  \label{L5_def}
\end{equation}%
by (\ref{g3out}),%
\begin{equation}
L_{6}=\left( xyC_{3}^{\alpha _{3}}C_{4}^{\alpha _{4}}C_{5}^{\alpha
_{5}}C_{6}^{\alpha _{6}}C_{7}^{\alpha _{7}},x,x,x\right) =\left(
xy,x,x,x\right)  \label{L6_def}
\end{equation}%
by (\ref{g2out}),%
\begin{equation}
L_{7}=(xyC_{3}^{\alpha _{3}}C_{4}^{\alpha _{4}}C_{5}^{\alpha
_{5}}C_{6}^{\alpha _{6}}C_{7}^{\alpha _{7}},x,xyC_{3}^{\alpha
_{3}}C_{4}^{\alpha _{4}}C_{5}^{\alpha _{5}}C_{6}^{\alpha _{6}}C_{7}^{\alpha
_{7}},xyC_{3}^{\alpha _{3}}C_{4}^{\alpha _{4}}C_{5}^{\alpha
_{5}}C_{6}^{\alpha _{6}}C_{7}^{\alpha _{7}})=  \label{L7_def}
\end{equation}%
\begin{equation*}
(xy,x,xy,xy)
\end{equation*}%
by (\ref{g2out}).

By (\ref{L3_def}) and (\ref{l_d}) we have that%
\begin{equation}
L_{3}=(xy,x)\left( xy,x,C_{3}^{\alpha _{3}}C_{4}^{\alpha _{4}}C_{5}^{\alpha
_{5}}\right) \left( C_{3}^{\alpha _{3}}C_{4}^{\alpha _{4}}C_{5}^{\alpha
_{5}},x\right) .  \label{L3decomp}
\end{equation}%
By (\ref{l_d}) the equality%
\begin{equation}
(xy,x)=(x,x)(x,x,y)(y,x)=(y,x)  \label{(xy,x)}
\end{equation}%
holds. By (\ref{g3out}) and (\ref{(xy,x)}) we conclude that%
\begin{equation}
\left( xy,x,C_{3}^{\alpha _{3}}C_{4}^{\alpha _{4}}C_{5}^{\alpha _{5}}\right)
=\left( xy,x,C_{3}^{\alpha _{3}}\right) =\left( (y,x),(y,x)^{\alpha
_{3}}\right) =1.  \label{(xy,x,C3)}
\end{equation}%
By (\ref{l_d_r_d}) we have that%
\begin{equation}
(C_{3}^{\alpha _{3}}C_{4}^{\alpha _{4}}C_{5}^{\alpha _{5}},x)=\left(
C_{3},x\right) ^{\alpha _{3}}(C_{4},x)^{\alpha _{4}}(C_{5},x)^{\alpha
_{5}}=C_{5}^{\alpha _{3}}C_{8}^{\alpha _{4}}C_{6}^{\alpha _{5}}.
\label{(C3C4C5,x)}
\end{equation}%
Hence, by (\ref{L3decomp}), (\ref{(xy,x)}), (\ref{(xy,x,C3)}) and (\ref%
{(C3C4C5,x)}) we have that%
\begin{equation}
L_{3}=C_{3}C_{5}^{\alpha _{3}}C_{8}^{\alpha _{4}}C_{6}^{\alpha _{5}}.
\label{L3}
\end{equation}

By (\ref{l_d}), (\ref{l_d_r_d}), (\ref{(xy,x)}) and (\ref{(xy,x,C3)}) we
conclude that%
\begin{equation}
\left( xyC_{3}^{\alpha _{3}},x\right) =\left( xy,x\right) \left(
xy,x,C_{3}^{\alpha _{3}}\right) \left( C_{3}^{\alpha _{3}},x\right)
=C_{3}C_{5}^{\alpha _{3}}.  \label{(xyC3^a3,x)}
\end{equation}%
By (\ref{L4_def}), (\ref{(xyC3^a3,x)}), (\ref{l_d_r_d}), (\ref{r_d}), (\ref%
{g2out}) and (\ref{g4_powers}), we have that%
\begin{equation*}
L_{4}=\left( C_{3}C_{5}^{\alpha _{3}},xyC_{3}^{\alpha _{3}}\right) =\left(
C_{3},xyC_{3}^{\alpha _{3}}\right) \left( C_{5},xyC_{3}^{\alpha _{3}}\right)
^{\alpha _{3}}=
\end{equation*}%
\begin{equation}
\left( C_{3},C_{3}^{\alpha _{3}}\right) (C_{3},xy)(C_{3},xy,C_{3}^{\alpha
_{3}})\left( C_{5},x\right) ^{\alpha _{3}}\left( C_{5},y\right) ^{\alpha
_{3}}.  \label{L4decomp}
\end{equation}%
By (\ref{r_d}) and (\ref{C8}) the equalities%
\begin{equation}
\left( C_{3},xy\right) =\left( C_{3},y\right) \left( C_{3},x\right) \left(
C_{3},x,y\right) =C_{4}C_{5}C_{8}  \label{(C3,xy)}
\end{equation}%
fulfills. Therefore, by (\ref{L4decomp}), (\ref{(C3,xy)}), (\ref{C8}) and
because $(C_{3},xy,C_{3}^{\alpha _{3}})\in \gamma _{5}\left( G\right) $ the%
\begin{equation}
L_{4}=C_{4}C_{5}C_{8}C_{6}^{\alpha _{3}}C_{8}^{\alpha
_{3}}=C_{4}C_{5}C_{6}^{\alpha _{3}}C_{8}^{\alpha _{3}+1}  \label{L4}
\end{equation}%
holds.

By (\ref{L5_def}), (\ref{(xyC3^a3,x)}) and (\ref{l_d_r_d}) we have that%
\begin{equation}
L_{5}=\left( C_{3}C_{5}^{\alpha _{3}},x\right) =\left( C_{3},x\right) \left(
C_{5},x\right) ^{\alpha _{3}}=C_{5}C_{6}^{\alpha _{3}}.  \label{L5}
\end{equation}

By (\ref{L6_def}), (\ref{L7_def}) and (\ref{g4_powers}) we can conclude that%
\begin{equation}
L_{6}=\left( y,x,x,x\right) =C_{6}  \label{L6}
\end{equation}%
and%
\begin{equation}
L_{7}=(y,x,x,x)(y,x,y,x)(y,x,x,y)(y,x,y,y)=C_{6}C_{7},  \label{L7}
\end{equation}%
because, by (\ref{C8}) and (\ref{Rxy1}), $(y,x,y,x)(y,x,x,y)=C_{8}^{2}=1$.

Therefore, by (\ref{xc(xcy)}), (\ref{L3}), (\ref{L4}), (\ref{L5}), (\ref{L6}%
), (\ref{L7}), (\ref{Rxy1}), we have that the left hand of (\ref{2X}) is
equal to%
\begin{equation*}
x\circ \left( x\circ y\right) =x^{2}yC_{3}^{\alpha _{3}}C_{4}^{\alpha
_{4}}C_{5}^{\alpha _{5}}C_{6}^{\alpha _{6}}C_{7}^{\alpha _{7}}\cdot
\end{equation*}%
\begin{equation*}
C_{3}^{\alpha _{3}}C_{5}^{\alpha _{3}^{2}}C_{8}^{\alpha _{3}\alpha
_{4}}C_{6}^{\alpha _{3}\alpha _{5}}\cdot C_{4}^{\alpha _{4}}C_{5}^{\alpha
_{4}}C_{6}^{\alpha _{3}\alpha _{4}}C_{8}^{\left( \alpha _{3}+1\right) \alpha
_{4}}\cdot C_{5}^{\alpha _{5}}C_{6}^{\alpha _{3}\alpha _{5}}\cdot
\end{equation*}%
\begin{equation*}
C_{6}^{\alpha _{6}}\cdot C_{6}^{\alpha _{7}}C_{7}^{\alpha _{7}}=
\end{equation*}%
\begin{equation}
x^{2}yC_{3}^{2\alpha _{3}}C_{5}^{\alpha _{3}^{2}+\alpha _{4}}C_{6}^{\alpha
_{3}\alpha _{4}+\alpha _{7}}C_{8}^{\alpha _{4}}.  \label{x*(x*y)}
\end{equation}

The right hand of (\ref{2X}) is equal to%
\begin{equation}
\left( x\circ x\right) \circ y=x^{2}\circ y=x^{2}yS_{3}^{\alpha
_{3}}S_{4}^{\alpha _{4}}S_{5}^{\alpha _{5}}S_{6}^{\alpha _{6}}S_{7}^{\alpha
_{7}}  \label{(xcx)cy}
\end{equation}%
where%
\begin{equation}
S_{3}=(y,x^{2}),  \label{S3_def}
\end{equation}%
\begin{equation}
S_{4}=\left( y,x^{2},y\right) ,  \label{S4_def}
\end{equation}%
\begin{equation}
S_{5}=\left( y,x^{2},x^{2}\right) ,  \label{S5_def}
\end{equation}%
\begin{equation}
S_{6}=\left( y,x^{2},x^{2},x^{2}\right) ,  \label{S6_def}
\end{equation}%
\begin{equation}
S_{7}=(y,x^{2},y,y).  \label{S7_def}
\end{equation}%
By (\ref{S3_def}) and (\ref{r_d}) we have that%
\begin{equation}
S_{3}=(y,x)(y,x)(y,x,x)=C_{3}^{2}C_{5}.  \label{S3}
\end{equation}%
By Lemma (\ref{C3_2}) $C_{3}^{2}\in \gamma _{4}\left( G\right) $, hence, by (%
\ref{S4_def}), (\ref{g4out}) and (\ref{C8}),%
\begin{equation}
S_{4}=\left( S_{3},y\right) =\left( C_{5},y\right) =C_{8}.  \label{S4}
\end{equation}%
By (\ref{S5_def}), (\ref{S3}), (\ref{g4out}), (\ref{g4_powers}) and (\ref%
{Rxy1}) we have that%
\begin{equation}
S_{5}=\left( S_{3},x^{2}\right) =\left( C_{5},x^{2}\right) =C_{6}^{2}=1.
\label{S5}
\end{equation}%
Also by (\ref{S6_def}) and (\ref{S5}) 
\begin{equation}
S_{6}=\left( S_{5},x^{2}\right) =1.  \label{S6}
\end{equation}%
By (\ref{S7_def}) and (\ref{S3})%
\begin{equation}
S_{7}=(S_{3},y,y)=1  \label{S7}
\end{equation}%
because $S_{3}\in \gamma _{3}\left( G\right) $. Therefore, by (\ref{(xcx)cy}%
), (\ref{S3}), (\ref{S4}), (\ref{S5}), (\ref{S6}) and (\ref{S7}) we have that%
\begin{equation}
\left( x\circ x\right) \circ y=x^{2}yC_{3}^{2\alpha _{3}}C_{5}^{\alpha
_{3}}C_{8}^{\alpha _{4}}  \label{(x*x)*y}
\end{equation}

By (\ref{2X}) we conclude from (\ref{x*(x*y)}) and (\ref{(x*x)*y}) that%
\begin{equation*}
x^{2}yC_{3}^{2\alpha _{3}}C_{5}^{\alpha _{3}^{2}+\alpha _{4}}C_{6}^{\alpha
_{3}\alpha _{4}+\alpha _{7}}C_{8}^{\alpha _{4}}=x^{2}yC_{3}^{2\alpha
_{3}}C_{5}^{\alpha _{3}}C_{8}^{\alpha _{4}}.
\end{equation*}%
We compare the exponents of the basic elements $C_{5}$ and $C_{6}$ in both
sides of this equality and deduce these two congruences:%
\begin{equation*}
\alpha _{3}^{2}+\alpha _{4}\equiv \alpha _{3}\left( \func{mod}2\right) ,
\end{equation*}%
\begin{equation*}
\alpha _{3}\alpha _{4}+\alpha _{7}\equiv 0\left( \func{mod}2\right) .
\end{equation*}%
When $\alpha _{3}\equiv 0\left( \func{mod}2\right) $ both when $\alpha
_{3}\equiv 1\left( \func{mod}2\right) $, we conclude from these congruences
that $\alpha _{4}\equiv 0\left( \func{mod}2\right) \ $and $\alpha _{7}\equiv
0\left( \func{mod}2\right) $. Therefore, the word $w_{\cdot }\left(
x,y\right) $ in the applicable system of words necessary has a form%
\begin{equation}
w_{\cdot }\left( x,y\right) =xyC_{3}^{\alpha _{3}}C_{5}^{\alpha
_{5}}C_{6}^{\alpha _{6}},  \label{w_r}
\end{equation}%
where $0\leq \alpha _{3}<4$, $\alpha _{5},\alpha _{6}\in \left\{ 0,1\right\} 
$.

Now we will compute the right hand of (\ref{2Y}) when $\circ $ is the verbal
operation defined by the word (\ref{w_r}):%
\begin{equation*}
\left( x\circ y\right) \circ y=w_{\cdot }\left( xyC_{3}^{\alpha
_{3}}C_{5}^{\alpha _{5}}C_{6}^{\alpha _{6}},y\right) =xyC_{3}^{\alpha
_{3}}C_{5}^{\alpha _{5}}C_{6}^{\alpha _{6}}yQ_{3}^{\alpha _{3}}Q_{5}^{\alpha
_{5}}Q_{6}^{\alpha _{6}}=
\end{equation*}%
\begin{equation}
=xy^{2}C_{3}^{\alpha _{3}}\left( C_{3}^{\alpha _{3}},y\right) C_{5}^{\alpha
_{5}}\left( C_{5}^{\alpha _{5}},y\right) C_{6}^{\alpha _{6}}Q_{3}^{\alpha
_{3}}Q_{5}^{\alpha _{5}}Q_{6}^{\alpha _{6}},  \label{(x*y)*y_b}
\end{equation}%
because $C_{6}^{\alpha _{6}}\in \gamma _{4}\left( G\right) $. Here%
\begin{equation}
Q_{3}=(y,xyC_{3}^{\alpha _{3}}C_{5}^{\alpha _{5}}C_{6}^{\alpha
_{6}})=(y,xyC_{3}^{\alpha _{3}}C_{5}^{\alpha _{5}}),  \label{Q3_def}
\end{equation}%
by (\ref{g4out});%
\begin{equation}
Q_{5}=\left( y,xyC_{3}^{\alpha _{3}}C_{5}^{\alpha _{5}}C_{6}^{\alpha
_{6}},xyC_{3}^{\alpha _{3}}C_{5}^{\alpha _{5}}C_{6}^{\alpha _{6}}\right)
=\left( y,xyC_{3}^{\alpha _{3}},xyC_{3}^{\alpha _{3}}\right) ,
\label{Q5_def}
\end{equation}%
by (\ref{g3out}); and%
\begin{equation*}
Q_{6}=\left( y,xyC_{3}^{\alpha _{3}}C_{5}^{\alpha _{5}}C_{6}^{\alpha
_{6}},xyC_{3}^{\alpha _{3}}C_{5}^{\alpha _{5}}C_{6}^{\alpha
_{6}},xyC_{3}^{\alpha _{3}}C_{5}^{\alpha _{5}}C_{6}^{\alpha _{6}}\right) =
\end{equation*}%
\begin{equation}
=\left( y,xy,xy,xy\right)  \label{Q6_def}
\end{equation}%
by (\ref{g2out}).

By (\ref{l_d_r_d}) we have that%
\begin{equation}
\left( C_{3}^{\alpha _{3}},y\right) =\left( C_{3},y\right) ^{\alpha
_{3}}=(y,x,y)^{\alpha _{3}}=C_{4}^{\alpha _{3}},  \label{(C3^a3,y)}
\end{equation}%
and%
\begin{equation}
\left( C_{5}^{\alpha _{5}},y\right) =\left( C_{5},y\right) ^{\alpha
_{5}}=\left( y,x,x,y\right) ^{\alpha _{5}}=C_{8}^{\alpha _{5}}
\label{(C5^a5,y)}
\end{equation}%
by (\ref{C8}).

We obtain the next equality from (\ref{r_d}), (\ref{l_d_r_d}), (\ref{Rxy1})
and Lemma \ref{g2}:%
\begin{equation*}
\left( y,xyC_{3}^{\alpha _{3}}\right) =(y,C_{3}^{\alpha
_{3}})(y,xy)(y,xy,C_{3}^{\alpha _{3}})=
\end{equation*}%
\begin{equation}
(y,C_{3})^{\alpha _{3}}(y,y)(y,x)(y,x,y)=C_{3}C_{4}^{\alpha _{3}+1}.
\label{(y,xyC3^a3)}
\end{equation}

After this we conclude from (\ref{Q3_def}), (\ref{r_d}), (\ref{(C5^a5,y)}), (%
\ref{(y,xyC3^a3)}), (\ref{C8}) and (\ref{Rxy1}) that%
\begin{equation}
Q_{3}=(y,xyC_{3}^{\alpha _{3}}C_{5}^{\alpha _{5}})=(y,C_{5}^{\alpha
_{5}})(y,xyC_{3}^{\alpha _{3}})(y,xyC_{3}^{\alpha _{3}},C_{5}^{\alpha
_{5}})=C_{3}C_{4}^{\alpha _{3}+1}C_{8}^{\alpha _{5}}.  \label{Q3}
\end{equation}

By (\ref{Q5_def}), (\ref{(y,xyC3^a3)}), (\ref{l_d_r_d}), (\ref{r_d}), (\ref%
{g4_powers}), (\ref{Rxy1}), (\ref{C8}) and because $\left(
C_{3},xy,C_{3}^{\alpha _{3}}\right) \in \gamma _{5}\left( G\right) $ we have
that%
\begin{equation*}
Q_{5}=\left( C_{3}C_{4}^{\alpha _{3}+1},xyC_{3}^{\alpha _{3}}\right) =\left(
C_{3},xyC_{3}^{\alpha _{3}}\right) \left( C_{4},xyC_{3}^{\alpha _{3}}\right)
^{\alpha _{3}+1}=
\end{equation*}%
\begin{equation*}
\left( C_{3},C_{3}^{\alpha _{3}}\right) \left( C_{3},xy\right) \left(
C_{3},xy,C_{3}^{\alpha _{3}}\right) \left( C_{4},xy\right) ^{\alpha
_{3}+1}=\left( C_{3},xy\right) \left( C_{4},xy\right) ^{\alpha _{3}+1}=
\end{equation*}%
\begin{equation}
\left( C_{3},y\right) \left( C_{3},x\right) \left( C_{3},x,y\right) \left(
C_{4},y\right) ^{\alpha _{3}+1}\left( C_{4},x\right) ^{\alpha _{3}+1}=
\label{Q5}
\end{equation}%
\begin{equation*}
C_{4}C_{5}C_{8}C_{7}^{\alpha _{3}+1}C_{8}^{\alpha
_{3}+1}=C_{4}C_{5}C_{7}^{\alpha _{3}+1}C_{8}^{\alpha _{3}}
\end{equation*}

From (\ref{g4_powers}), (\ref{C8}) and (\ref{Rxy1}) we conclude that%
\begin{equation*}
Q_{6}=\left( y,xy,xy,xy\right) =
\end{equation*}%
\begin{equation}
\left( y,x,x,x\right) \left( y,x,x,y\right) \left( y,x,y,x\right) \left(
y,x,y,y\right) =C_{6}C_{7}.  \label{Q6}
\end{equation}

Therefore, by (\ref{(x*y)*y_b}), (\ref{(C3^a3,y)}), (\ref{(C5^a5,y)}), (\ref%
{Q3}), (\ref{Q5}), (\ref{Q6}), (\ref{Rxy1}) and by Lemma \ref{g2} 
\begin{equation*}
\left( x\circ y\right) \circ y=
\end{equation*}%
\begin{equation*}
xy^{2}C_{3}^{\alpha _{3}}C_{4}^{\alpha _{3}}C_{5}^{\alpha _{5}}C_{8}^{\alpha
_{5}}C_{6}^{\alpha _{6}}C_{3}^{\alpha _{3}}C_{4}^{\alpha _{3}^{2}+\alpha
_{3}}C_{8}^{\alpha _{3}\alpha _{5}}C_{4}^{\alpha _{5}}C_{5}^{\alpha
_{5}}C_{7}^{\alpha _{3}\alpha _{5}+\alpha _{5}}C_{8}^{\alpha _{3}\alpha
_{5}}C_{6}^{\alpha _{6}}C_{7}^{\alpha _{6}}=
\end{equation*}%
\begin{equation}
xy^{2}C_{3}^{2\alpha _{3}}C_{4}^{\alpha _{3}^{2}+\alpha _{5}}C_{7}^{\alpha
_{3}\alpha _{5}+\alpha _{5}+\alpha _{6}}C_{8}^{\alpha _{5}}.
\label{(x*y)*y_f}
\end{equation}

Now we will compute the left side of (\ref{2Y}). As above $\circ $ is the
verbal operation defined by the word (\ref{w_r}). We have that%
\begin{equation}
x\circ \left( y\circ y\right) =w_{\cdot }\left( x,y^{2}\right)
=xy^{2}U_{3}^{\alpha _{3}}U_{5}^{\alpha _{5}}U_{6}^{\alpha _{6}},
\label{x*(y*y)_b}
\end{equation}%
where%
\begin{equation}
U_{3}=(y^{2},x),  \label{U3_def}
\end{equation}%
\begin{equation}
U_{5}=\left( y^{2},x,x\right) ,  \label{U5_def}
\end{equation}%
\begin{equation}
U_{6}=\left( y^{2},x,x,x\right) .  \label{U6_def}
\end{equation}%
By (\ref{U3_def}), (\ref{l_d}) and by Lemma \ref{g2} we conclude that%
\begin{equation}
U_{3}=(y,x)(y,x,y)(y,x)=C_{3}^{2}C_{4}.  \label{U3}
\end{equation}%
By (\ref{U5_def}), (\ref{U3}), (\ref{l_d_r_d}) and, because, by Lemma \ref%
{C3_2}, $\left( C_{3}^{2},x\right) \in \gamma _{5}\left( G\right) $, we
deduce that%
\begin{equation}
U_{5}=\left( U_{3},x\right) =\left( C_{3}^{2},x\right) \left( C_{4},x\right)
=\left( C_{4},x\right) =C_{8}.  \label{U5}
\end{equation}%
Also by (\ref{U6_def}), (\ref{g4_powers}) and (\ref{Rxy1}) we have that%
\begin{equation}
U_{6}=\left( y,x,x,x\right) ^{2}=C_{6}^{2}=1.  \label{U6}
\end{equation}%
Therefore, by (\ref{x*(y*y)_b}), (\ref{U3}), (\ref{U5}), (\ref{U6}) and by
Lemma \ref{g2} we obtain%
\begin{equation}
x\circ \left( y\circ y\right) =xy^{2}C_{3}^{2\alpha _{3}}C_{4}^{\alpha
_{3}}C_{8}^{\alpha _{5}}.  \label{x*(y*y)_f}
\end{equation}

By (\ref{2Y}) we conclude from (\ref{(x*y)*y_f}) and (\ref{x*(y*y)_f}) that%
\begin{equation*}
xy^{2}C_{3}^{2\alpha _{3}}C_{4}^{\alpha _{3}^{2}+\alpha _{5}}C_{7}^{\alpha
_{3}\alpha _{5}+\alpha _{5}+\alpha _{6}}C_{8}^{\alpha
_{5}}=xy^{2}C_{3}^{2\alpha _{3}}C_{4}^{\alpha _{3}}C_{8}^{\alpha _{5}}.
\end{equation*}%
We compare the exponents of the basic elements $C_{4}$ and $C_{7}$ in both
sides of this equality and deduce these two congruences:%
\begin{equation*}
\alpha _{3}^{2}+\alpha _{5}\equiv \alpha _{3}\left( \func{mod}2\right) ,
\end{equation*}%
\begin{equation*}
\alpha _{3}\alpha _{5}+\alpha _{5}+\alpha _{6}\equiv 0\left( \func{mod}%
2\right) .
\end{equation*}%
When $\alpha _{3}\equiv 0\left( \func{mod}2\right) $ both when $\alpha
_{3}\equiv 1\left( \func{mod}2\right) $, we conclude from these congruences
that $\alpha _{5}\equiv 0\left( \func{mod}2\right) \ $and $\alpha _{6}\equiv
0\left( \func{mod}2\right) $. Therefore, the word $w_{\cdot }\left(
x,y\right) $ in the applicable system of words necessary has a form%
\begin{equation}
w_{\cdot }\left( x,y\right) =xyC_{3}^{\alpha _{3}},  \label{aw}
\end{equation}%
where $0\leq \alpha _{3}<4$.
\end{proof}

From Propositions \ref{AN1} and \ref{AN2} we conclude that in the variety $%
\Theta $ the applicable system of words can have only these four forms%
\begin{equation}
W_{\alpha }=\left\{ w_{1},w_{-1}\left( x\right) =x^{-1},w_{\cdot }\left(
x,y\right) =xyC_{3}^{\alpha }\right\} ,  \label{syst_words_n}
\end{equation}%
where $0\leq \alpha <4$.

\section{Applicable systems of words. Sufficient conditions\label{asw_sc}}

\setcounter{equation}{0}

We will prove in this section that all the systems of words mentioned in (%
\ref{syst_words_n}) are applicable.

It is obvious that the system of words $W_{0}$ is applicable (see Subsection %
\ref{bijections_words}).

In the begging of this section we will prove that the systems of words $%
W_{1} $ and $W_{2}$ are applicable and after this we will conclude that the
systems of words $W_{3}$ is applicable.

\subsection{System of words $W_{1}$}

Now we consider the system of words $W_{1}$. In this system of words $%
w_{\cdot }\left( x,y\right) =xy\left( y,x\right) =yx$. We denote by $%
\underset{1}{\circ }$ the verbal operation defined by the word $w_{\cdot
}\left( x,y\right) =yx$. We will prove that for every $G\in \Theta $ the
universal algebra $G_{W_{1}}^{\ast }$ is also a group of the variety $\Theta 
$.

It is clear that for every $G\in \Theta $ and every $x\in G$ the identities%
\begin{equation*}
x\underset{1}{\circ }1=1\underset{1}{\circ }x=x
\end{equation*}%
hold.

\begin{proposition}
\label{ass_1}The operation $\underset{1}{\circ }$ is an associative
operation.
\end{proposition}

\begin{proof}
For every $G\in \Theta $ and every $x,y,z\in G$ we have that $\left( x%
\underset{1}{\circ }y\right) \underset{1}{\circ }z=z\left( yx\right) =\left(
zy\right) x=x\underset{1}{\circ }\left( y\underset{1}{\circ }z\right) $.
\end{proof}

We denote for every $m\in 
\mathbb{Z}
$ by $x^{\underset{1}{\circ }m}$ the degree $m$ defined system of words $%
W_{1}$ of the element $x\in G$, where $G\in \Theta $. It is clear that $x^{%
\underset{1}{\circ }m}=x^{m}$, so for every $G\in \Theta $ and every $x\in G$
the identities%
\begin{equation*}
x\underset{1}{\circ }x^{\underset{1}{\circ }-1}=x^{\underset{1}{\circ }-1}%
\underset{1}{\circ }x=1
\end{equation*}%
and%
\begin{equation*}
x^{\underset{1}{\circ }4}=1
\end{equation*}%
hold.

For every $G\in \Theta $ and every $x,y\in G$ we will denote $\left(
x,y\right) _{1}=x^{-1}\underset{1}{\circ }y^{-1}\underset{1}{\circ }x%
\underset{1}{\circ }y$.

\begin{proposition}
\label{metab_1}For every $G\in \Theta $ and every $x_{1},x_{2},x_{3},x_{4}%
\in G$ the identity%
\begin{equation*}
\left( \left( x_{1},x_{2}\right) _{1},\left( x_{3},x_{4}\right) _{1}\right)
_{1}=1
\end{equation*}%
holds.
\end{proposition}

\begin{proof}
We have that%
\begin{equation}
\left( x,y\right) _{1}=y^{-1}x^{-1}\underset{1}{\circ }yx=yxy^{-1}x^{-1}=%
\left( y^{-1},x^{-1}\right) =\left( x^{-1},y^{-1}\right) ^{-1}.
\label{commut_1}
\end{equation}%
Therefore%
\begin{equation*}
\left( \left( x_{1},x_{2}\right) _{1},\left( x_{3},x_{4}\right) _{1}\right)
_{1}=\left( \left( x_{2}^{-1},x_{1}^{-1}\right) ,\left(
x_{4}^{-1},x_{3}^{-1}\right) \right) _{1}=
\end{equation*}%
\begin{equation*}
\left( \left( x_{4}^{-1},x_{3}^{-1}\right) ^{-1},\left(
x_{2}^{-1},x_{1}^{-1}\right) ^{-1}\right) =\left( \left(
x_{3}^{-1},x_{4}^{-1}\right) ,\left( x_{1}^{-1},x_{2}^{-1}\right) \right) =1.
\end{equation*}
\end{proof}

\begin{proposition}
\label{nilp4_1}For every $G\in \Theta $ and every $%
x_{1},x_{2},x_{3},x_{4},x_{5}\in G$ the identity%
\begin{equation*}
\left( \left( \left( \left( x_{1},x_{2}\right) _{1},x_{3}\right)
_{1},x_{4}\right) _{1},x_{5}\right) _{1}=1
\end{equation*}%
holds.
\end{proposition}

\begin{proof}
By (\ref{commut_1}) we have that%
\begin{equation}
\left( \left( \left( x_{1},x_{2}\right) _{1},\ldots \right)
_{1},x_{n}\right) _{1}=\left( \left( \left( x_{1}^{-1},x_{2}^{-1}\right)
,\ldots \right) ,x_{n}^{-1}\right) ^{-1}  \label{com_ind_1}
\end{equation}%
holds when $n=2$. We suppose that (\ref{com_ind_1}) holds. So, by (\ref%
{commut_1}), we have that%
\begin{equation*}
\left( \left( \left( \left( x_{1},x_{2}\right) _{1},\ldots \right)
_{1},x_{n}\right) _{1},x_{n+1}\right) _{1}=
\end{equation*}%
\begin{equation*}
\left( \left( \left( \left( x_{1},x_{2}\right) _{1},\ldots \right)
_{1},x_{n}\right) _{1}^{-1},x_{n+1}^{-1}\right) ^{-1}=\left( \left( \left(
\left( x_{1}^{-1},x_{2}^{-1}\right) ,\ldots \right) ,x_{n}^{-1}\right)
,x_{n+1}^{-1}\right) ^{-1}.
\end{equation*}%
Therefore, we proved (\ref{com_ind_1}) for every $n\geq 2$. In particular,
for every $G\in \Theta $ and every $x_{1},x_{2},x_{3},x_{4},x_{5}\in G$ we
have that%
\begin{equation*}
\left( \left( \left( \left( x_{1},x_{2}\right) _{1},x_{3}\right)
_{1},x_{4}\right) _{1},x_{5}\right) _{1}=\left( x_{1}^{-1},\ldots
,x_{5}^{-1}\right) ^{-1}=1\text{.}
\end{equation*}
\end{proof}

Therefore, we proved that for every $G\in \Theta $ the universal algebra $%
G_{W_{1}}^{\ast }$ is also a group of the variety $\Theta $.

In particular, we have that $F_{W_{1}}^{\ast }\in \Theta $ for every $F\in 
\mathrm{Ob}\Theta ^{0}$. So, for every $F=F_{\Theta }\left( X\right) \in 
\mathrm{Ob}\Theta ^{0}$ there exists a homomorphism $s_{F}^{\left( 1\right)
}:F\rightarrow F_{W_{1}}^{\ast }$ such that $s_{F\mid X}^{\left( 1\right) }=%
\mathrm{id}_{X}$.

\begin{proposition}
\label{asw_1}The system of words $W_{1}$ is an applicable system of words.
\end{proposition}

\begin{proof}
For every $F=F_{\Theta }\left( X\right) \in \mathrm{Ob}\Theta ^{0}$ and
every $a,b\in F$ we have that%
\begin{equation*}
\left( s_{F}^{\left( 1\right) }\right) ^{2}\left( ab\right) =s_{F}^{\left(
1\right) }\left( s_{F}^{\left( 1\right) }\left( a\right) \underset{1}{\circ }%
s_{F}^{\left( 1\right) }\left( b\right) \right) =s_{F}^{\left( 1\right)
}\left( s_{F}^{\left( 1\right) }\left( b\right) s_{F}^{\left( 1\right)
}\left( a\right) \right) =
\end{equation*}%
\begin{equation*}
\left( s_{F}^{\left( 1\right) }\right) ^{2}\left( b\right) \underset{1}{%
\circ }\left( s_{F}^{\left( 1\right) }\right) ^{2}\left( a\right) =\left(
s_{F}^{\left( 1\right) }\right) ^{2}\left( a\right) \left( s_{F}^{\left(
1\right) }\right) ^{2}\left( b\right) .
\end{equation*}%
So, $\left( s_{F}^{\left( 1\right) }\right) ^{2}:F\rightarrow F$ is a
homomorphism. The equality $\left( s_{F\mid X}^{\left( 1\right) }\right)
^{2}=\mathrm{id}_{X}$ holds, hence $\left( s_{F}^{\left( 1\right) }\right)
^{2}=\mathrm{id}_{F}$. Therefore, $s_{F}^{\left( 1\right) }$ is a bijection.
It means that $s_{F}^{\left( 1\right) }$ is an isomorphism. Hence $W_{1}$ is
a subject of Definition \ref{asw}.
\end{proof}

\subsection{System of words $W_{2}$}

We will prove in this subsection that the system of words $W_{2}$ is an
applicable system of words. $w_{\cdot }\left( x,y\right) =xy\left(
y,x\right) ^{2}$ in this system of words. As above we denote by $\underset{2}%
{\circ }$ the verbal operation defined by the word $w_{\cdot }\left(
x,y\right) =xy\left( y,x\right) ^{2}$. We will prove that for every $G\in
\Theta $ the universal algebra $G_{W_{2}}^{\ast }$ is also a group of the
variety $\Theta $.

It is clear that for every $G\in \Theta $ and every $x\in G$ the identities%
\begin{equation*}
x\underset{2}{\circ }1=1\underset{2}{\circ }x=x
\end{equation*}%
hold.

\begin{proposition}
\label{ass_2}The operation $\underset{2}{\circ }$ is an associative
operation.
\end{proposition}

\begin{proof}
For every $G\in \Theta $ and every $x,y,z\in G$ we have that%
\begin{equation*}
\left( x\underset{2}{\circ }y\right) \underset{2}{\circ }z=xy\left(
y,x\right) ^{2}\underset{2}{\circ }z=xy\left( y,x\right) ^{2}z\left(
z,xy\left( y,x\right) ^{2}\right) ^{2}=
\end{equation*}%
\begin{equation*}
xyz\left( y,x\right) ^{2}\left( z,xy\right) ^{2}=xyz\left( y,x\right)
^{2}\left( z,y\right) ^{2}\left( z,x\right) ^{2}.
\end{equation*}%
In this computation we use Corollary \ref{theta} from Theorem \ref{freeGroup}%
, Lemma \ref{C3_2} and (\ref{q_r_d}). By similar computation we conclude that%
\begin{equation*}
x\underset{2}{\circ }\left( y\underset{2}{\circ }z\right) =x\underset{2}{%
\circ }yz\left( z,y\right) ^{2}=xyz\left( z,y\right) ^{2}\left( yz\left(
z,y\right) ^{2},x\right) ^{2}=
\end{equation*}%
\begin{equation*}
xyz\left( z,y\right) ^{2}\left( yz,x\right) ^{2}=xyz\left( z,y\right)
^{2}\left( y,x\right) ^{2}\left( z,x\right) ^{2}.
\end{equation*}
\end{proof}

As above we denote for every $m\in 
\mathbb{Z}
$ by $x^{\underset{2}{\circ }m}$ the degree $m$ defined system of words $%
W_{2}$ of the element $x\in G$, where $G\in \Theta $. And just as before, it
is clear that $x^{\underset{2}{\circ }m}=x^{m}$, so for every $G\in \Theta $
and every $x\in G$ the identities%
\begin{equation*}
x\underset{2}{\circ }x^{\underset{2}{\circ }-1}=x^{\underset{2}{\circ }-1}%
\underset{2}{\circ }x=1
\end{equation*}%
and%
\begin{equation*}
x^{\underset{2}{\circ }4}=1
\end{equation*}%
hold.

For every $G\in \Theta $ and every $x,y\in G$ we will denote $\left(
x,y\right) _{2}=x^{-1}\underset{2}{\circ }y^{-1}\underset{2}{\circ }x%
\underset{2}{\circ }y$.

\begin{proposition}
\label{commut_2}For every $G\in \Theta $ and every $x,y\in G$ the equality%
\begin{equation*}
\left( x,y\right) _{2}=\left( x,y\right)
\end{equation*}%
holds.
\end{proposition}

\begin{proof}
We have by Corollary \ref{theta} from Theorem \ref{freeGroup} and Lemma \ref%
{C3_2}, by (\ref{g4out}), (\ref{q_l_d}), (\ref{q_r_d}), (\ref{q_i_l}), (\ref%
{q_i_r}) and (\ref{exponent}) that%
\begin{equation*}
\left( x,y\right) _{2}=x^{-1}y^{-1}\left( y^{-1},x^{-1}\right) ^{2}\underset{%
2}{\circ }xy\left( y,x\right) ^{2}=
\end{equation*}%
\begin{equation*}
x^{-1}y^{-1}\left( y^{-1},x^{-1}\right) ^{2}xy\left( y,x\right) ^{2}\left(
xy\left( y,x\right) ^{2},x^{-1}y^{-1}\left( y^{-1},x^{-1}\right) ^{2}\right)
^{2}=
\end{equation*}%
\begin{equation*}
\left( x,y\right) \left( y,x\right) ^{4}\left( xy,x^{-1}y^{-1}\right)
^{2}=\left( x,y\right) \left( xy,x^{-1}y^{-1}\right) ^{2}=
\end{equation*}%
\begin{equation*}
\left( x,y\right) \left( x,y^{-1}\right) ^{2}\left( y,x^{-1}\right)
^{2}=\left( x,y\right) .
\end{equation*}
\end{proof}

\begin{corollary}
For every $G\in \Theta $ and every $x_{1},x_{2},x_{3},x_{4},x_{5}\in G$ the
identities%
\begin{equation*}
\left( \left( x_{1},x_{2}\right) _{2},\left( x_{3},x_{4}\right) _{2}\right)
_{2}=1
\end{equation*}%
and%
\begin{equation*}
\left( \left( \left( \left( x_{1},x_{2}\right) _{2},x_{3}\right)
_{2},x_{4}\right) _{2},x_{5}\right) _{2}=1
\end{equation*}%
hold.
\end{corollary}

Therefore, we proved that for every $G\in \Theta $ the universal algebra $%
G_{W_{2}}^{\ast }$ is also a group of the variety $\Theta $.

In particular, we have that $F_{W_{2}}^{\ast }\in \Theta $ for every $F\in 
\mathrm{Ob}\Theta ^{0}$. So, for every $F=F_{\Theta }\left( X\right) \in 
\mathrm{Ob}\Theta ^{0}$ there exists a homomorphism $s_{F}^{\left( 2\right)
}:F\rightarrow F_{W_{2}}^{\ast }$ such that $s_{F\mid X}^{\left( 2\right) }=%
\mathrm{id}_{X}$.

\begin{proposition}
The system of words $W_{2}$ is an applicable system of words.
\end{proposition}

\begin{proof}
It is clear, that for every $G\in \Theta $, every $a\in G$ and every $b\in
\gamma _{4}\left( G\right) $ the equality $a\underset{2}{\circ }b=ab$ holds.
Therefore, for every $F=F_{\Theta }\left( X\right) \in \mathrm{Ob}\Theta
^{0} $ and every $a,b\in F$ we have by Proposition \ref{commut_2}, by
Corollary \ref{theta} from Theorem \ref{freeGroup} and Lemma \ref{C3_2}, and
by (\ref{exponent}) that%
\begin{equation*}
\left( s_{F}^{\left( 2\right) }\right) ^{2}\left( ab\right) =s_{F}^{\left(
2\right) }\left( s_{F}^{\left( 2\right) }\left( a\right) \underset{2}{\circ }%
s_{F}^{\left( 2\right) }\left( b\right) \right) =
\end{equation*}%
\begin{equation*}
s_{F}^{\left( 2\right) }\left( s_{F}^{\left( 2\right) }\left( a\right)
s_{F}^{\left( 2\right) }\left( b\right) \left( s_{F}^{\left( 2\right)
}\left( b\right) ,s_{F}^{\left( 2\right) }\left( a\right) \right)
^{2}\right) =
\end{equation*}%
\begin{equation*}
\left( s_{F}^{\left( 2\right) }\right) ^{2}\left( a\right) \underset{2}{%
\circ }\left( s_{F}^{\left( 2\right) }\right) ^{2}\left( b\right) \underset{2%
}{\circ }\left( \left( s_{F}^{\left( 2\right) }\right) ^{2}\left( b\right)
,\left( s_{F}^{\left( 2\right) }\right) ^{2}\left( a\right) \right) _{2}^{2}=
\end{equation*}%
\begin{equation*}
\left( s_{F}^{\left( 2\right) }\right) ^{2}\left( a\right) \left(
s_{F}^{\left( 2\right) }\right) ^{2}\left( b\right) \left( \left(
s_{F}^{\left( 2\right) }\right) ^{2}\left( b\right) ,\left( s_{F}^{\left(
2\right) }\right) ^{2}\left( a\right) \right) ^{2}\underset{2}{\circ }\left(
\left( s_{F}^{\left( 2\right) }\right) ^{2}\left( b\right) ,\left(
s_{F}^{\left( 2\right) }\right) ^{2}\left( a\right) \right) ^{2}=
\end{equation*}%
\begin{equation*}
\left( s_{F}^{\left( 2\right) }\right) ^{2}\left( a\right) \left(
s_{F}^{\left( 2\right) }\right) ^{2}\left( b\right) \left( \left(
s_{F}^{\left( 2\right) }\right) ^{2}\left( b\right) ,\left( s_{F}^{\left(
2\right) }\right) ^{2}\left( a\right) \right) ^{4}=\left( s_{F}^{\left(
2\right) }\right) ^{2}\left( a\right) \left( s_{F}^{\left( 2\right) }\right)
^{2}\left( b\right) .
\end{equation*}%
So, $\left( s_{F}^{\left( 2\right) }\right) ^{2}:F\rightarrow F$ is a
homomorphism. And, as in Proposition \ref{asw_1}, this completes the proof.
\end{proof}

\subsection{System of words $W_{3}$}

We proved that the systems of words $W_{1}$ and $W_{2}$ are applicable. So,
there exist $\mathcal{C}^{-1}\left( W_{1}\right) =\Phi _{1},\mathcal{C}%
^{-1}\left( W_{2}\right) =\Phi _{2}\in \mathfrak{S}$. Hence, the applicable
systems of words $\mathcal{C}\left( \Phi _{2}\Phi _{1}\right) $ we can
obtain by (\ref{der_veb_opr_prod}), where $s_{F_{\omega }}^{\Phi
_{1}}=s_{F_{\omega }}^{\left( 1\right) }$, $s_{F_{\omega }}^{\Phi
_{2}}=s_{F_{\omega }}^{\left( 2\right) }$, $\omega \in \Omega =\left\{
1,-1,\cdot \right\} $.

\begin{proposition}
The equality $\mathcal{C}\left( \Phi _{2}\Phi _{1}\right) =W_{3}$ holds.
\end{proposition}

\begin{proof}
$s_{F_{\omega }}^{\left( 1\right) }$ and $s_{F_{\omega }}^{\left( 2\right) }$
fix the words $w_{1}=1$ and $w_{-1}\left( x\right) =x^{-1}$. So, it's enough
to compute $s_{G}^{\left( 2\right) }s_{G}^{\left( 1\right) }\left( xy\right) 
$, where $G=F_{\Theta }\left( x,y\right) $. This word will be $w_{\cdot
}\left( x,y\right) $ in the applicable system of words $\mathcal{C}\left(
\Phi _{2}\Phi _{1}\right) $. $s_{G}^{\left( 1\right) }:G\rightarrow
G_{W_{1}}^{\ast }$ and $s_{G}^{\left( 2\right) }:G\rightarrow
G_{W_{2}}^{\ast }$ are isomorphisms and they fix the generators $x$ and $y$.
Therefore, by (\ref{exponent}),%
\begin{equation*}
s_{G}^{\left( 2\right) }s_{G}^{\left( 1\right) }\left( xy\right)
=s_{G}^{\left( 2\right) }\left( x\underset{1}{\circ }y\right) =s_{G}^{\left(
2\right) }\left( yx\right) =y\underset{2}{\circ }x=yx\left( x,y\right)
^{2}=xy\left( x,y\right) =xy\left( y,x\right) ^{3}.
\end{equation*}
\end{proof}

We conclude from this proposition that $W_{3}$ is an applicable systems of
words.

\section{Group $\mathfrak{S\cap Y}$ and group $\mathfrak{A/Y}$}

\setcounter{equation}{0}

We conclude from Section \ref{asw_sc} that group $\mathfrak{S}$ contains $4$
elements: automorphisms $\Phi _{\alpha }=\mathcal{C}^{-1}\left( W_{\alpha
}\right) $, where $0\leq \alpha <4$.

\begin{theorem}
The $\mathfrak{S\cap Y=}\left\{ \Phi _{0},\Phi _{1}\right\} $ and $%
\left\vert \mathfrak{A/Y}\right\vert =2$ hold.
\end{theorem}

\begin{proof}
By Criterion \ref{inner_stable} the automorphism $\Phi _{\alpha }$ is inner
if and only if for every $F\in \mathrm{Ob}\Theta ^{0}$ there exists an
isomorphism $c_{F}^{\left( \alpha \right) }:F\rightarrow F_{W_{\alpha
}}^{\ast }$, which fulfills condition (\ref{commutmor}) for every $A,B\in 
\mathrm{Ob}\Theta ^{0}$ and every $\psi \in \mathrm{Mor}_{\Theta ^{0}}\left(
A,B\right) $. By Proposition \ref{centr_func}, it means, in particular, that
there exists $c(x)\in F_{\Theta }(x)$ such that the equality (\ref%
{commutfunc}) holds.

On the other hand isomorphisms $c_{F}^{\left( \alpha \right) }$, where $F\in 
\mathrm{Ob}\Theta ^{0}$, must be bijections. The group $F_{\Theta }(x)$
contains only $4$ elements: $c_{i}\left( x\right) =x^{i}$, where $0\leq i<4$%
. For every $F\in \mathrm{Ob}\Theta ^{0}$ we consider mappings $\left(
c_{i}\right) _{F}:F\rightarrow F$ defined for every $f\in F$ by formula $%
\left( c_{i}\right) _{F}(f)=c_{i}(f)=f^{i}$. It is easy to check that $%
\mathrm{im}\left( c_{0}\right) _{F_{\Theta }(x)}=\left\{ 1\right\} \neq
F_{\Theta }(x)$ and $\mathrm{im}\left( c_{2}\right) _{F_{\Theta
}(x)}=\left\{ 1,x^{2}\right\} \neq F_{\Theta }(x)$. When $i=1$ or $i=3$ then
the mappings $\left( c_{i}\right) _{F}:F\rightarrow F$, such that for every $%
f\in F$ the $\left( c_{1}\right) _{F}(f)=f$ and $\left( c_{3}\right)
_{F}(f)=f^{3}=f^{-1}$ holds, are bijections, because for every $F\in \mathrm{%
Ob}\Theta ^{0}$ we have that $\left( c_{1}\right) _{F}=\mathrm{id}_{F}$ and $%
\left( \left( c_{3}\right) _{F}\right) ^{2}=\mathrm{id}_{F}$.

We will denote $\left( c_{1}\right) _{F}=c_{F}^{\left( 0\right) }$ and $%
\left( c_{3}\right) _{F}=c_{F}^{\left( 1\right) }$ for every $F\in \mathrm{Ob%
}\Theta ^{0}$. It is clear that for every $F\in \mathrm{Ob}\Theta ^{0}$ the
mapping $c_{F}^{\left( 0\right) }=\mathrm{id}_{F}:F\rightarrow
F_{W_{0}}^{\ast }$ is an isomorphism, because $F=F_{W_{0}}^{\ast }$.

Also for every $F\in \mathrm{Ob}\Theta ^{0}$ and every $a,b\in F$ the
equality%
\begin{equation*}
c_{F}^{\left( 1\right) }\left( a\right) \underset{1}{\circ }c_{F}^{\left(
1\right) }\left( b\right) =a^{-1}\underset{1}{\circ }b^{-1}=b^{-1}a^{-1}=%
\left( ab\right) ^{-1}=c_{F}^{\left( 1\right) }\left( ab\right)
\end{equation*}%
holds. Therefore, $c_{F}^{\left( 1\right) }:F\rightarrow F_{W_{1}}^{\ast }$
is an isomorphism. By Proposition \ref{centr_func} we have that condition (%
\ref{commutmor}) holds for isomorphisms $c_{F}^{\left( 0\right)
}:F\rightarrow F_{W_{0}}^{\ast }$ and isomorphisms $c_{F}^{\left( 1\right)
}:F\rightarrow F_{W_{1}}^{\ast }$ ($F\in \mathrm{Ob}\Theta ^{0}$). It proves
that $\Phi _{0},\Phi _{1}\in \mathfrak{S\cap Y}$.

We will denote $F_{\Theta }(x,y)=G$. We have that%
\begin{equation*}
c_{G}^{\left( 0\right) }\left( x\right) \underset{2}{\circ }c_{G}^{\left(
0\right) }\left( y\right) =x\underset{2}{\circ }y=xy\left( y,x\right)
^{2}\neq c_{G}^{\left( 0\right) }\left( xy\right) =xy
\end{equation*}%
and%
\begin{equation*}
c_{G}^{\left( 1\right) }\left( x\right) \underset{2}{\circ }c_{G}^{\left(
1\right) }\left( y\right) =x^{-1}\underset{2}{\circ }y^{-1}=x^{-1}y^{-1}%
\left( y^{-1},x^{-1}\right) ^{2}\neq c_{G}^{\left( 1\right) }\left(
xy\right) =\left( xy\right) ^{-1}=y^{-1}x^{-1}
\end{equation*}%
because%
\begin{equation*}
xyx^{-1}y^{-1}\left( y^{-1},x^{-1}\right) ^{2}=\left( x^{-1},y^{-1}\right)
\left( y^{-1},x^{-1}\right) ^{2}=\left( y^{-1},x^{-1}\right) \neq 1.
\end{equation*}%
Therefore, neither mapping $c_{G}^{\left( 0\right) }$ nor mapping $%
c_{G}^{\left( 1\right) }$ are isomorphisms $F\rightarrow F_{W_{2}}^{\ast }$,
so the automorphism $\Phi _{2}\notin \mathfrak{S\cap Y}$. The Lagrange
Theorem argument completes the proof.
\end{proof}

\section{Open problem}

As we said in the Section \ref{Intr}, we can't conclude from fact that the
group $\mathfrak{A/Y}$ is not trivial that in our variety $\Theta $ the
difference between geometric and automorphic equivalences exists. We must
construct a specific example of the two groups from the variety $\Theta $,
such that they are automorphically equivalent but are not geometrically
equivalent. This construction is yet the open problem.

\section{Acknowledgement}

The first author acknowledge the support of Coordena\c{c}\~{a}o de Aperfei%
\c{c}oamento de Pessoal de N\'{\i}vel Superior - CAPES (Coordination for the
Improvement of Higher Education Personnel, Brazil).

We are thankful to Prof. E. Aladova for her important remarks, which helped
a lot in writing this article.

We acknowledge Prof. A. I. Reznokov from St.-Petersburg State University,
which provide to the authors the copy of \cite{Sanov}.

\end{document}